\documentclass[a4paper]{amsart}
\usepackage[all]{xy}
\usepackage{bbm}
\usepackage{enumerate}
\usepackage{amsmath} 
\usepackage{amsfonts}
\usepackage{amssymb}

\newcommand{\setB}{\mathbb{B}}
\newcommand{\setN}{\mathbb{N}} 
\newcommand{\setZ}{\mathbb{Z}} 
\newcommand{\setQ}{\mathbb{Q}}
\newcommand{\setR}{\mathbb{R}}
\newcommand{\setS}{\mathbb{S}}
\newcommand{\calA}{\mathcal{A}}
\newcommand{\calH}{\mathcal{H}}
\newcommand{\calO}{\mathcal{O}}
\newcommand{\calR}{\mathcal{R}}
\newcommand{\calS}{\mathcal{S}}
\newcommand{\calV}{\mathcal{V}}
\newcommand{\calZ}{\mathcal{Z}}
\newcommand{\trace}[1]{\mbox{Tr}\,{#1}}
\newcommand{\eps}{\varepsilon }
\newcommand{\tensor}{\otimes}
\newcommand{\quot}[2]{{#1}\slash{#2}}
\newcommand{\abs}[1]{\left\vert #1 \right\vert }
\newcommand{\norm}[1]{\left\Vert #1 \right\Vert }
\newcommand{\conj}[1]{\overline{#1 }}
\newcommand{\jnoc}[1]{\underline{#1 }}
\newcommand{\slra}[1]{\stackrel{#1}{\longrightarrow}}
\newcommand{\incc}[1]{\left[{#1 }\right]}
\newcommand{\inoo}[1]{\left({#1 }\right)}
\providecommand{\boxx}{\,\square\,}
\providecommand{\nabla}{\,\triangledown\,}
\newcommand{\sklein}[1]{\mbox{\begin{scriptsize}$#1$\end{scriptsize}}}
\newcommand{\scalar}[1]{\left\langle #1 \right\rangle }
\renewcommand{\smash}{\wedge}

\DeclareMathOperator{\colim}{colim}
\DeclareMathOperator{\sign}{sign}
\DeclareMathOperator{\id}{id}

\newtheorem{theorem}{Theorem}[section]
\newtheorem{proposition}[theorem]{Proposition}
\newtheorem{lemma}[theorem]{Lemma}
\newtheorem{corollary}[theorem]{Corollary}
\theoremstyle{definition}
\newtheorem{definition}[theorem]{Definition}

\title[Equivariant Lefschetz theory]{Equivariant Lefschetz and Fuller indices via topological intersection theory}
\author{Philipp Wruck}
\thanks{The author was supported by a grant from the Deutsche Forschungsgemeinschaft (DFG)}
\begin{document}

\begin{abstract}
 For a compact Lie group $G$, we use $G$-equivariant Poincar\'{e} duality for ordinary $RO(G)$-graded homology to define an equivariant intersection product, the dual of the 
 equivariant cup product. Using this, we give a homological construction of the equivariant Lefschetz number and a simple proof of the equivariant Lefschetz fixed point theorem. 
 With similar techniques, an equivariant Fuller index with values in the rationalized Burnside ring is constructed. 
\end{abstract}
\maketitle
\setcounter{section}{-1}
\section{Introduction}
 The classical Lefschetz number can be defined in various ways, all of which are equivalent in the category of smooth compact manifolds. All the facts in the following discussion
 can be found e.g. in Bredon's introductory book \cite{bredon}. Notation and terminology that has not been introduced will be explained more thoroughly in the main part of the 
 paper. 

 The first definition uses ordinary homology. For a continuous map $f:X\to X$ of a finite CW complex, one can define 
 \[
  L^{hom}(f)=\trace{H_*(f;\setQ)}=\sum_{k=0}^\infty(-1)^kH_k(f;\setQ)\in\setZ,
 \]
 the trace of the map $f$ induces in rational homology. One can prove that it actually takes values in the integers.

If $f:M\to M$ is a smooth self map of a smooth compact manifold and the map $(\id, f):M\to M\times M$ is transverse to the diagonal, one can define
\[
 L^{smooth}(f)=\sum_{x\in Fix(f)}\sign\det T_xf\in\setZ.
\]
The transversality condition ensures that this sum is finite, and this definition can be extended to arbitrary continuous maps. 

Finally, if $M$ is smooth, compact and orientable of dimension $n$, one can use the Poincar\'{e} duality isomorphism $P_*:H^*(M\times M;\setZ)\to H_{2n-*}(M\times M;\setZ)$. If 
$\calO_\Delta$ denotes the image of the fundamental class of $M$ under the diagonal embedding $M\to M\times M$, $\calO_\Gamma$ the image of the fundamental class of $M$ under the 
map $(\id, f):M\to M\times M$ and $M\times M$ is oriented by the cross product of the fundamental class of $M$ with itself, one defines
\[
 L^{geom}(f)=\eps(P_{2n}(P_n^{-1}(\calO_\Delta)\cup P_n^{-1}(\calO_\Gamma)))\in\setZ,
\]
where $\eps:H_0(M\times M)\to\setZ$ is the augmentation. 

A fundamental result is the Lefschetz fixed point theorem which states that non-vanishing of the Lefschetz number of $f$ implies existence of a fixed point. More generally,
one can show that the Lefschetz number only depends on the restriction of the map $f$ to an arbitrary neighbourhood of its fixed points. Then under some genericity assumptions,
one can show that it is non-trivial around an isolated fixed point. 

This result is apparent using $L^{smooth}$, but using $L^{geom}$ gives a better understanding of the geometric reasons. The intersection product of two fundamental classes is
the fundamental class of the intersection of these manifolds, again assuming some transversality conditions. Thus, the intersection product of the diagonal class with the graph
class of $f$ is the fundamental class of the manifold of fixed points, and the theorem follows.

Using the geometric interpretation of the Poincar\'{e} dual of the cup product as related to physical intersections therefore is very useful and it is desirable to have a similar
intuition when there is a group action involved. 

The basic idea is to use ordinary equivariant homology, cohomology and the respective dual theories constructed by Costenoble and Waner in \cite{costenoble}. The essential 
definition will assume that the manifold is $G$-orientable in the sense of \cite{may2}, and is a $V$-manifold in the sense that the tangential space of $M$ at $x$ is isomorphic 
to a fixed $G$-representation $V$ as a $G_x$-representation. This can be generalized to the case where $M$ is any $G$-manifold by some standard constructions of fixed point 
theory.

The special case of $V$-manifolds allows the theory to be developed along the lines of the non-equivariant constructions. The Poincar\'{e} dual of the equivariant cup product, the 
equivariant intersection product, will be defined and it will be shown that under some rather restrictive assumptions, the geometric interpretation mentioned before is still 
valid in this case.

The definition of an equivariant Lefschetz number then proceeds as in the non-equivariant case, yielding an element in the so called Burnside ring $A(G)$ of $G$ for which we can
prove an equivariant Lefschetz theorem. For a closed subgroup $H$ of $G$, we have the fixed point homomorphism $\eta_H:A(G)\to A(W(H))$, $W(H)$ being the Weyl group of $H$.
\setcounter{section}{5}
\begin{theorem}
 Let $M$ be a compact orientable $V$-manifold. Then if a $G$-map $f:M\to M$ has no fixed point of orbit type at least $(H)$, we have $\eta_H(L_G(f))=0$.
\end{theorem}

We will show that in case $G$ is finite, our definition of an equivariant Lefschetz number agrees with Definition (4.1) of \cite{rosenberg} of the ``equivariant Lefschetz class 
with values in the Burnside ring''.

We then proceed to compute the equivariant Lefschetz number from local data around the fixed point set, which can be seen as a generalization of Theorem \ref{thm:lefschetz}. 
A map $f:V\to V$ of a $G$-representation to itself induces a map $\setS^V\to\setS^V$ via the Pontryagin-Thom construction. We denote the stable homotopy class of this induced map
by $Deg_G(f)$. We have an induction map $t^G_H:A(H)\to A(G)$ for subgroups $H$ of $G$. The result then is the following.

\begin{theorem}
 Let $M$ be a $G$-manifold and $f:M\to M$ a $G$-map with finitely many $G$-orbits of fixed points $Gx_1,\dots, Gx_n$. Assume that $\id-N_{x_i}f$ has no eigenvalue of unit modulus
 for $i=1, \dots, n$, where $N_xf$ is the component of $T_xf$ normal to the orbit $Gx$. Then the equivariant Lefschetz number $L_G(f)$ is given as
 \[
  L_G(f)=\sum_{i=1}^nt^G_{G_{x_i}}\left(Deg_{G_{x_i}}(\id-N_{x_i}f)\right).
\]
\end{theorem}
\setcounter{section}{6}
\setcounter{theorem}{0}
The homological techniques used to define the equivariant Lefschetz number can be generalized to find other homotopy invariants. In particular, they can be used to define an 
equivariant homological Fuller index. This is a homotopy invariant for flows, where periodic orbits of the flow take the role of fixed points of maps. We will construct an 
equivariant generalization of the homological Fuller index of \cite{franzosa}, which will take values in the rationalized Burnside ring. It will be shown that this index behaves 
nicely under restriction to group fixed points and is a homological version of the index constructed by the author in \cite{wruck}, where methods of dynamical system theory were 
used. This section results in the following.

\begin{theorem}
 The equivariant Fuller index $F_G$ is a $G$-homotopy invariant of a flow $\varphi$ with respect to an isolated set $C$ of periodic points, and has the following properties.
 \begin{enumerate}[i)]
  \item It takes values in the rationalized Burnside ring $A(G)\tensor\setQ$.
  \item If $C$ consists of finitely many periodic orbits $\gamma_1,\dots, \gamma_n$ and $\varphi_i$ is the flow $\varphi$, restricted to an isolating neighbourhood of the 
  orbit $\gamma_i$, then
  \[
   F_G(\varphi)=\sum_{i=1}^nF_G(\varphi_i).
  \]

  \item If $\varphi$ has a single periodic orbit of multiplicity $m$, then $F_G(\varphi)=L_G(P^m)\tensor\frac 1m\in A(G)\tensor\setQ$, where $P$ is an equivariant Poincar\'{e}
  map for the orbit, considered with multiplicity one.
 
  \item If $\eta_H(F_G(\varphi))\neq0$, then $\varphi$ has a periodic orbit of orbit type at least $(H)$.
 \end{enumerate}
\end{theorem}
\setcounter{section}{0}

The paper is organized as follows. In section one, we review the construction of equivariant ordinary homology, cohomology and the dual theories. Our main reference is 
\cite{costenoble}, which sadly is only available on the arxiv so far. So we also refer to the published reference \cite{alaska} at several places.
In section two, we discuss restriction behaviour and products in these homology theories, still following \cite{costenoble} closely. In section three, we start to develop some
new material. We start with the basic duality theory as developed in \cite{costenoble} and show that duality behaves well with respect to products. This already allows us
to define the equivariant Lefschetz number using equivariant intersection theory. We then turn in section four to restriction properties of the Lefschetz number, both to fixed 
sets and to subgroups, which turn out to be key properties in what is to follow. In section five, we prove the equivariant Lefschetz fixed point theorem \ref{thm:lefschetz}
and its generalization Theorem \ref{thm:lefschetz2}, and establish the fact that the equivariant Lefschetz number equals the L\"{u}ck-Rosenberg Lefschetz class for finite $G$. 

In section six, we use similar techniques to those of section three and four, but this time for flows, to define an equivariant Fuller index. We prove that this index behaves 
well with respect to restriction, from which it follows that this index equals the dynamical index of \cite{wruck}. This section therefore can be seen as an equivariant 
generalization of \cite{franzosa} and as a complement to \cite{wruck}.
 

\section{Ordinary Equivariant Homology}
We have to establish some conventions at the beginning. Throughout the paper, unless otherwise stated, $G$ will be a compact Lie group. Subgroups of $G$ are always assumed to be 
closed. A $G$-space $X$ is a pointed topological space $X$ with a left $G$-action $G\times X\to X,\;(g,x)\mapsto gx$. The base point of $X$ is fixed by $G$. A $G$-manifold is 
assumed to be smooth with a smooth $G$-action and is not assumed to be pointed. For homological considerations, we will work with pointed $G$-spaces throughout and we will add a 
disjoint base point to unbased spaces $X$, denoting the result by $X_+$.

We will use several standard constructions from equivariant topology, all of which can be found in \cite{bredongroups}. Most notable is the twisted product $X\times_HY$ of a 
right (unpointed) $H$-space $X$ with a left (unpointed) $H$-space $Y$. This is defined to be the quotient space $\quot{(X\times Y)}H$, where $H$ acts as $h(x, y)=(xh, hy)$. The 
twisted product becomes a left $G$-space provided $X$ carries a left $G$-action such that $g(xh)=(gx)h$ for all $g\in G$, $h\in H$, $x\in X$. Then $g[x,y]=[gx,y]$ is a well 
defined action of $G$, where $[x,y]$ denotes the class of $(x, y)$ in $X\times_HY$. Similarly we have a twisted smash product $X\smash_HY$ of pointed spaces, which is 
obtained from $X\smash Y$ by identifying $[xh, y]$ with $[x, hy]$. As before, this carries a left $G$-action, provided $X$ carries a left $G$-action with the aforementioned 
compatibility assumption. In almost all cases, $X$ will be the $G$-space $G$, acting by left and right translations on itself. 

In general, when we have to denote a class of an element $x$ under a standard quotient map, such as $G\mapsto\quot GH$, $X\times Y\mapsto X\smash Y$ and the like, we will use the 
notation $[x]$ for that class, as long as no confusion is probable. 

The construction of equivariant ordinary homology rests on the definition of $G$-CW(V) complexes, where $V$ is any orthogonal representation of $G$. Such complexes generalize
ordinary $G$-complexes. Most notably, the dual cell structure of a $G$-CW structure on a $V$-manifold is a $G$-$CW(V)$ structure.
 
We will not go into the details of the construction and refer to \cite{costenoble} or \cite{alaska} instead. Let it be said that each $V$ gives rise to a cellular homology theory 
graded on $\setZ$, and these, for varying representations, can be pasted together to give a theory graded on $RO(G)$. One can then use $CW(V)$ approximations to define the 
homology theory on the category of all $G$-spaces. We specify what we want to understand by $RO(G)$ in the following definition.

\begin{definition}
 Let $G$ be a compact Lie group. Let $I$ be the set of $G$-isomorphism classes of orthogonal irreducible $G$-representations. We choose a representant $V_i$ of $i\in I$. In 
 particular, we can assume that the underlying space of $V_i$ is $\setR^{n_i}$ for some $n_i\in\setN$. The free abelian group on the elements $V_i$ is denoted by $RO(G)$ and is 
 called the real representation ring of $G$. The ring structure on $RO(G)$ is induced by the tensor product of representations.
\end{definition}

If $V$ is any orthogonal $G$-representation, it is isomorphic to a direct sum of the form $\bigoplus_{i\in I}V_i^{k_i}$, where almost all $k_i$ are equal to zero. We fix one 
isomorphism for each representation $V$, taking identities whenever possible. In the following, whenever two representations pop up which turn out to be isomorphic, we silently 
assume that we fix the isomorphism to be given by the two particular isomorphisms with the representing element in $RO(G)$. We also introduce the notation $\abs{V}$ for the real 
dimension of the vector space $V$. This will simplify matters when we restrict from equivariant $RO(G)$-graded theories to integer graded theories.

Throughout the paper, the Burnside ring $A(G)$ of a compact Lie group $G$ will play a prominent role. Recall that the normalizer of a subgroup $H$ of $G$ is defined as
\[
 N(H)=\{g\in G\;|\;gHg^{-1}\subseteq H\}.
\]
It is a subgroup of $G$ and $H$ is normal in $N(H)$. The quotient $\quot{N(H)}{H}$ is denoted by $W(H)$ and is called the Weyl group of $H$. If $X$ is any $G$-space, the fixed
point space $X^H$ is a $W(H)$-space in a natural way.

The Burnside ring can be defined to be the free abelian group on the set of $G$-orbits of the form $\quot GH$, where $H\subseteq G$ is a subgroup with $W(H)$ finite. The product 
is induced by Cartesian product of $G$-orbits. Of fundamental importance is the fact that the Burnside ring is isomorphic to $\pi_0^G(\setS^0)$, the zeroth equivariant stable 
homotopy group of the equivariant sphere spectrum. It is in this disguise that we will encounter $A(G)$ most of the time. For further details, we refer to \cite{may} or 
\cite{tomdieck}.

We quickly introduce the notion of a Mackey functor, since these functors will be our coefficient systems.

\begin{definition}
 Let $G$ be a compact Lie group. The stable orbit category $\hat{\calO}_G$ is the category whose objects are the orbits $\quot GH$ for subgroups $H$ of $G$, and whose morphisms 
 are the stable $G$-maps between orbits. $\hat{\calO}_G$ is an additive category, therefore we can define a $G$-Mackey functor to be an additive functor from $\hat{\calO}_G$ to 
 the category $\calA b$ of abelian groups.
\end{definition}

It is common to define a Mackey functor to be a contravariant additive functor, but we will need the covariant version as well. We therefore make the following convention. A 
contravariant Mackey functor is labelled with an overlining, e.g. $\conj{T}:\hat{\calO}_G\to\calA$. A covariant Mackey functor is labelled with an underlining, e.g. 
$\underline{S}:\hat{\calO}_G\to\calA b$.

The most important Mackey functors for our purposes are the functors
\[
 \conj{A}_{\quot GH}:\hat{\calO}_G\to\calA b,\;\conj{A}_{\quot GH}(\quot GK)=\{{\quot GK}_+,{\quot GH}_+\}_G,
\]
\[
 \underline{A}^{\quot GH}:\hat{\calO}_G\to\calA b,\;\underline{A}^{\quot GH}(\quot GK)=\{{\quot GH}_+,{\quot GK}_+\}_G.
\]
Here and in the following, we will denote the set (and in fact abelian group) of stable $G$-maps between the pointed $G$-spaces $X$ and $Y$ by $\{X,Y\}_G$.

The Mackey functor $\conj{A}_{\quot GG}$ is also called the Burnside Mackey functor, since $\conj{A}_{\quot GG}(\quot GK)\cong A(K)$. This functor takes the role that $\setZ$ 
plays non-equivariantly.

Mackey functors can be restricted both to subgroups and to fixed sets. 

To define restriction to subgroups, let $K\subseteq G$ be a subgroup. Then there is the induction functor
\[
 i^G_K:\hat{\calO}_K\to\hat{\calO}_G,\;\quot KL\mapsto G\times_K\quot KL\cong\quot GL.
\]
For $G$-Mackey functors $\jnoc{S},\conj{T}$, the restriction to $K$ is defined as the pull-back via $i^G_K$, $\jnoc{S}\big|K={i^G_K}^*(\jnoc{S})=\jnoc{S}\circ i^G_K$ and 
$\conj{T}|K={i^G_K}^*(\conj{T})=\conj{T}\circ i^G_K$. 

If $K\subseteq G$ is normal, we have the $\quot GK$-Mackey functor $\jnoc{S}^K$ defined as follows. Let $\jnoc{S}_K(\quot GH)$ be the subgroup of $\jnoc{S}(\quot GH)$, generated
by the images of $\jnoc{S}(\quot GL)$ under maps $\jnoc{S}(\varphi)$ with $\varphi:\quot GL\to\quot GH$ and $K\nsubseteq L$. This is a sub-Mackey functor of $\jnoc{S}$ in the 
evident way and we define
\[
 \jnoc{S}^K(\quot{(\quot GK)}{(\quot HK)})=\quot{\jnoc{S}(\quot GH)}{\jnoc{S}_K(\quot GH)}.
\]
If $K$ is not normal, we define $\jnoc{S}^K$ to be the functor obtained by first restricting to the normalizer $N(K)$ of $K$ and then applying the preceeding construction. 
$\conj{T}^K$ for a contravariant Mackey functor is defined completely analogously.

The most important property of restriction of Mackey functors for our purposes is the fact that $\conj{A}_{\quot GG}\big|H\cong\conj{A}_{\quot HH}$ and 
$\conj{A}_{\quot GG}^K\cong\conj{A}_{\quot{W(K)}{W(K)}}$. This can easily be calculated from the definitions, or can be found in \cite{costenoble}.

We now state a basic theorem for $RO(G)$-graded ordinary homology and cohomology. It is similar to Theorem 1.2.5 of \cite{costenoble} and Theorem A of \cite{costenoble2}.
We denote with $\calS_G$ the category of $G$-spectra in the sense of Lewis, May and Steinberger \cite{may}. $\calR_G$ is the category with objects $RO(G)$. A morphism $V\to W$ is a stable 
homotopy class $\setS^V\to\setS^W$, induced by an isometric $G$-isomorphism $V\to W$. 

\begin{theorem}\label{thm:ordinaryproperties}
 Let $\conj{T}$ be a contravariant Mackey functor, $\underline{S}$ be a covariant Mackey functor. There are functors 
 \[
  H^G_*(\;\cdot\;;\underline{S}):\calS_G\times \calR_G\times \calR_G\to\calA b,\;(X, V, W)\mapsto H^G_{V-W}(X;\underline{S})
 \]
 and
 \[
  H^*_G(\;\cdot\;;\conj{T}):\calS_G\times\calR_G\times\calR_G\to\calA b,\;(X, V, W)\mapsto H_G^{V-W}(X;\conj{T}).
 \]
 These functors are represented by $G$-spectra $H\jnoc{S}$ and $H\conj{T}$, respectively. The corresponding dual theories are denoted by
 \[
  \calH_G^*(\;\cdot\;;\underline{S}):\calS_G\times\calR_G\times\calR_G\to\calA b,\;(X, V, W)\mapsto\calH_G^{V-W}(X;\underline{S})
 \]
 and
 \[
  \calH^G_*(\;\cdot\;;\conj{T}):\calS_G\times\calR_G\times\calR_G\to\calA b,\;(X, V, W)\mapsto\calH^G_{V-W}(X;\conj{T}).
 \]
 and are called dual ordinary cohomology and dual ordinary homology, respectively. These four theories possess the following properties.
\begin{enumerate}[(i)]
 \item For fixed $\alpha=V-W$, the functors $H^G_\alpha(\;\cdot\;;\underline{S})$ and $\calH^G_\alpha(\;\cdot\;;\conj{T})$ are exact on cofibre sequences and send wedges to 
 direct sums. The functors $H_G^\alpha(\;\cdot\;;\conj{T})$ and $\calH_G^\alpha(\;\cdot\;;\jnoc{S})$ are exact on cofibre sequences and send wedges to products.

 \item For $\alpha=V-W$ and $Z$ a $G$-representation, there are natural suspension isomorphisms
 \[
  \sigma_Z:H^G_\alpha(X;\jnoc{S})\to H^G_{\alpha+Z}(\Sigma^ZX;\jnoc{S})
 \]
  and
 \[
  \sigma_Z:H_G^\alpha(X;\conj{T})\to H_G^{\alpha+Z}(\Sigma^ZX;\conj{T}).
 \]
 Both of these satisfy $\sigma_Z\circ\sigma_{Z'}=\sigma_{Z'\oplus Z}$. Similar statements hold for the dual theories.

 \item If $n=V-W$ is an integer, the ordinary theories coincide with Bredon homology and Bredon cohomology. They satisfy the dimension axiom in the form that
 \[
  H_n^G(-_+;\jnoc{S})\cong\begin{cases}
                              \jnoc{S} & n=0\\
                              0        & \mbox{ else}
                             \end{cases}
 \]
 and
 \[
  H_G^n(-_+;\conj{T})\cong\begin{cases}
                              \conj{T} & n=0\\
                              0        & \mbox{ else}
                             \end{cases}
 \]
 as functors of the stable orbit category.

 For the dual theories, there holds
 \[
  \calH_n^G(D(-_+);\conj{T})\cong\begin{cases}
                              \conj{T} & n=0\\
                              0        & \mbox{ else}
                             \end{cases}
 \]
 and
 \[
  \calH^n_G(D(-_+);\jnoc{S})\cong\begin{cases}
                              \jnoc{S} & n=0\\
                              0        & \mbox{ else},
                             \end{cases}
 \]
where $D$ denotes the Spanier-Whitehead dual functor, assigning to a $G$-spec\-trum $X$ its dual $F(X,\setS^0)$ (compare \cite{may}, Definition III.3.3).

\item There are natural Wirthm\"{u}ller isomorphisms. For $K\subseteq G$ a subgroup, $\alpha=V-W$ and a $K$-spectrum $X$,
\[
 H^G_\alpha(G_+\wedge_KX;\jnoc{S})\cong H^K_{\alpha|K}(X;\jnoc{S}|K),
\]
\[
 H_G^\alpha(G_+\wedge_KX;\conj{T})\cong H_K^{\alpha|K}(X;\conj{T}|K),
\]
If $L$ denotes the tangential space at $[e]$ in $\quot GK$, then there are natural dual Wirthm\"{u}ller isomorphisms
\[
 \calH^G_\alpha(G_+\wedge_K(\setS^{-L}\wedge X);\conj{T})\cong\calH^K_{\alpha|K}(X;\conj{T}|K)
\]
and
\[
 \calH_G^\alpha(G_+\wedge_K(\setS^{-L}\wedge X);\jnoc{S})\cong\calH_K^{\alpha|K}(X;\jnoc{S}|K)
\] 
\item If $K\subseteq G$ is a closed normal subgroup, $\alpha=V-W$, and $\Phi^K$ denotes the geometric fixed point construction (Definition 1.5.6 of \cite{costenoble}), there is 
a natural restriction map
\[
 H_\alpha^G(X;\jnoc{S})\to H^{\quot GK}_{\alpha^K}(\Phi^K(X);\jnoc{S}^K),
\]
\[
 H^\alpha_G(X;\conj{T})\to H_{\quot GK}^{\alpha^K}(\Phi^K(X);\conj{T}^K),
\]
and restriction commutes with the suspension isomorphisms. 

For the dual theories, there are natural restrictions
\[
 \calH^G_\alpha(X;\conj{T})\to\calH_{\alpha^K}^{\quot GK}(\Phi^K(X);\conj{T}^K),
\]
\[
 \calH_G^\alpha(X;\jnoc{S})\to\calH^{\alpha^K}_{\quot GK}(\Phi^K(X);\jnoc{S}^K).
\]
On the subcategory of $G$-spaces, we can drop the $\Phi^K$ and take $X^K$ instead of $\Phi^K(X)$.
\end{enumerate}
\end{theorem}

The Wirthm\"{u}ller isomorphisms allow us to define restriction maps with respect to subgroups. We recall from \cite{may}, II.6 that the Spanier-Whitehead dual of an orbit 
${\quot GK}_+$ is given by the spectrum $G_+\wedge_K\setS^{-L}$, where $L$ as above is the tangential representation at $[e]\in\quot GK$.

We define
\begin{eqnarray*}
 H_\alpha^G(X;\jnoc{S})&\to&H_\alpha^G(D({\quot GK}_+)\wedge X;\jnoc{S})\\
                       &\cong&H_\alpha^G(G_+\wedge_K\setS^{-L}\wedge X;\jnoc{S})\\
                       &\cong&H_{\alpha|K}^K(\setS^{-L}\wedge X;\jnoc{S}|K)\\
                       &\cong&H_{(\alpha|K)+L}^K(X;\jnoc{S}|K),
\end{eqnarray*}
where the first map is induced by the dual to the projection $\quot GK\to *$ and the last is suspension.
Similarly,
\begin{eqnarray*}
 H^\alpha_G(X;\conj{T})&\to&H^\alpha_G(G_+\wedge_KX;\conj{T})\\
                       &\cong&H^{\alpha|K}_K(X;\conj{T}|K),
\end{eqnarray*}
where the first map is induced by the projection $\quot GK\to *$,
\begin{eqnarray*}
 \calH_\alpha^G(X;\conj{T})&\to&\calH_\alpha^G(D({\quot GK}_+)\wedge X;\jnoc{S})\\
                           &\cong&\calH_\alpha^G(G_+\wedge_K(\setS^{-L}\wedge X);\jnoc{S})\\
                           &\cong&\calH_{\alpha|K}^K(X;\jnoc{S}|K)
\end{eqnarray*}
and
\begin{eqnarray*}
 \calH_G^\alpha(X;\conj{T})&\to&\calH_G^\alpha(G_+\wedge_KX;\conj{T})\\
                           &\cong&\calH_K^{\alpha|K}(\setS^L\wedge X;\conj{T}|K)\\
                           &\cong&\calH_K^{\alpha|K-L}(X;\conj{T}|K).
\end{eqnarray*}
Most notable are the shifts in dimension for ordinary homology and ordinary dual cohomology, which of course vanish if $G$ is finite.


\section{Products}
The main tool to define products in equivariant homology and cohomology is a product of Mackey functors. We state the following result from \cite{costenoble}.
\begin{proposition}
 Let $\conj{S}, \conj{T}$ be two contravariant $G$-Mackey functors. Then there exists a box product $\conj{S}\boxtimes\conj{T}$, which is a $G\times G$-Mackey functor. The 
 restriction to the diagonal subgroup $G\subseteq G\times G$ yields an internal box product $\conj{S}\boxx\conj{T}$, a $G$-Mackey functor. The product has the following 
 properties.
 \begin{enumerate}[(i)]
  \item For $G$-spaces $E, F$, there is a natural isomorphism 
  \[
   \{-\,,E\}_G\boxtimes\{-\,,F\}_G\cong\{-\,,E\smash F\}_{G\times G}.
  \]
  \item $\conj{A}_{\quot GG}\boxx\conj{T}\cong\conj{T}$ for every $G$-Mackey functor $\conj{T}$.
 \end{enumerate}

There is also a product of contra- and covariant Mackey functors, $\conj{T}\nabla\jnoc{S}$. The corresponding properties are:
\begin{enumerate}[(i)]
 \item On $G$-orbits,
  \[
   \conj{A}_{\quot GL}\nabla\jnoc{A}^{\quot GK}\cong\{{\quot GK}_+,-_+\,\wedge{\quot GL}_+\}_G
  \]
 
 \item $\conj{A}_{\quot GG}\nabla\jnoc{S}\cong\jnoc{S}$ for every $G$-Mackey functor $\jnoc{S}$.
\end{enumerate}
\end{proposition}

It is particularly easy to define the products using the representing spectra of the theories as given by Theorem \ref{thm:ordinaryproperties}. The following is 
another summary of results from \cite{costenoble}, mainly sections 1.5 and 1.6. For a $G$-spectrum $X$, we recall the definition of the equivariant homotopy group functor
\[
 \conj{\pi}_n^G(X):\hat{\calO}\to\calA b,\;\conj{\pi}_n^G(X)(\quot GH)=\pi_n^K(X)=\{\Sigma^n{\quot GH}_+, X\}_G
\]
and the equivariant dual homotopy group functor
\[
 \jnoc{\tau}_n^G(X):\hat{\calO}\to\calA b,\;\jnoc{\tau}_n^G(X)(\quot GH)=\tau_n^K(X)=\{\setS^n,{\quot GH}_+\smash X\}_G.
\]

\begin{proposition}\label{prop:eilenbergspectra}
 Let $\conj{T}$ and $\conj{U}$ be contravariant Mackey functors, $\jnoc{S}$ be a covariant Mackey functor. The representing spectra $H\conj{T}$, $H\conj{U}$ and $H\jnoc{S}$ of
 ordinary homology and cohomology have the following properties.
 \begin{enumerate}[i)]
  \item
  \[
  \conj{\pi}_n^G(H\conj{T})\cong\begin{cases}
                                  \conj{T}&n=0\\
                                  0&\mbox{ else}
                                 \end{cases}
 \]
 
 \item
 \[
  \jnoc{\tau}_n^G(H\jnoc{S})\cong\begin{cases}
                                  \jnoc{S}&n=0\\
                                         0&\mbox{ else}
                                 \end{cases}
 \]

 \item Properties i) and ii) determine the spectra uniquely up to $G$-homotopy equivalence.

 \item For a subgroup $K$ of $G$, the $K$-spectrum $H\conj{T}\big|K$ represents $\conj{T}\big|K$ and the $W(K)$-spectrum $(H\conj{T})^K$ represents $\conj{T}^K$.

 \item For a subgroup $K$ of $G$ and $L$ the tangential representation at $[e]\in\quot GK$, the $K$-spectrum $\Sigma^LH\jnoc{S}$ represents $\jnoc{S}\big|K$.
 Furthermore, there is a $W(K)$-map $\Phi^K(H\jnoc{S})\to H(\jnoc{S}^K)$ which induces an isomorphism in zeroth equivariant dual homotopy.

 \item There is a $G$-map $H\conj{T}\smash H\conj{U}\to H(\conj{T}\Box\conj{U})$ which induces an isomorphism in zeroth equivariant homotopy.

 \item There is a $G$-map $H\conj{T}\smash H\jnoc{S}\to H(\conj{T}\nabla\jnoc{S})$ which induces an isomorphism in zeroth equivariant dual homotopy.  
 \end{enumerate} 
\end{proposition}

Parts vi) and vii) will be used to manufacture the products and pairings we will use later.

\begin{definition}
 Let $X, Y, Z$ be pointed $G$-spaces, $\alpha, \beta\in RO(G)$ and let $d:Z\to X\smash Y$ be a $G$-map. We define cohomology cross products
  \[
   \times:H^\alpha_G(X;\conj{T})\tensor H^\beta_G(Y;\conj{U})\to H^{\alpha+\beta}_G(X\smash Y;\conj{T}\Box\conj{U})
  \]
  \[
   \times:H^\alpha_G(X;\conj{T})\tensor\calH^\beta_G(Y;\jnoc{S})\to\calH^{\alpha+\beta}_G(X\smash Y;\conj{T}\nabla\jnoc{S})
  \]
  as follows. For $\xi:X\to\setS^\alpha\smash H\conj{T}$, $\eta:Y\to\setS^\beta\smash H\conj{U}$, the product is represented by
  \[
   X\smash Y\slra{\xi\smash\eta}\setS^\alpha\smash H\conj{T}\smash\setS^\beta\smash H\conj{U}\slra{\cong}\setS^{\alpha+\beta}\smash(H\conj{T}\smash H\conj{U})\to\setS^{\alpha+\beta}\smash H(\conj{T}\Box\conj{U}),
  \]
  where the last map is induced by the map from Proposition \ref{prop:eilenbergspectra} vi).

  For $\xi:X\to\setS^\alpha\smash H\conj{T}$ and $\eta:Y\to\setS^\beta\smash H\jnoc{S}$, the product is represented by
  \[
   X\smash Y\slra{\xi\smash\eta}\setS^\alpha\smash H\conj{T}\smash\setS^\beta\smash H\jnoc{S}\slra{\cong}\setS^{\alpha+\beta}\smash(H\conj{T}\smash H\jnoc{S})\to\setS^{\alpha+\beta}\smash H(\conj{T}\nabla\jnoc{S}),
  \]
  the last map now being induced by Proposition \ref{prop:eilenbergspectra} vii).
  
  We define equivariant cup products
  \[
   \cup_d:H^\alpha_G(X;\conj{T})\tensor H^\beta_G(Y;\conj{U})\to H^{\alpha+\beta}_G(Z;\conj{T}\Box\conj{U}),
  \]
  \[
   \cup_d:H^\alpha_G(X;\conj{T})\tensor\calH^\beta_G(Y;\jnoc{S})\to\calH^{\alpha+\beta}_G(Z;\conj{T}\nabla\jnoc{S})
  \]
  by following up the cross products with the map induced by $d$ in cohomology.
 
  We define pairings
  \[
   \scalar{-\,,\,-}_d:H^\alpha_G(Y; \conj{T})\tensor H_\beta^G(Z;\jnoc{S})\to H_{\beta-\alpha}^G(X;\conj{T}\nabla\jnoc{S}),
  \]
  \[
   \scalar{-\,,\,-}_d:\calH^\alpha_G(Y; \jnoc{S})\tensor\calH^G_\beta(Z;\conj{T})\to H_{\beta-\alpha}^G(X;\conj{T}\nabla\jnoc{S}),
  \]
  \[
   \scalar{-\,,\,-}_d:H^\alpha_G(Y; \conj{T})\tensor\calH^G_\beta(Z;\conj{U})\to\calH_{\beta-\alpha}^G(X;\conj{T}\Box\conj{U}),
  \]
  as follows. For $\xi:Y\to\setS^\alpha\smash H\conj{T}$ and $y:\setS^\beta\to Z\smash H\jnoc{S}$, their pairing is defined to be represented by
  \[
   \setS^\beta\slra{y}Z\smash H\jnoc{S}\slra{d\smash\id}X\smash Y\smash H\jnoc{S}\slra{\id\smash\xi\smash\id}X\smash\setS^\alpha\smash H\conj{T}\smash H\jnoc{S}\to
   \setS^\alpha\smash X\smash H(\conj{T}\nabla\jnoc{S}),
  \]
  where the last map again is induced by Proposition \ref{prop:eilenbergspectra} vii). The other two pairings are defined analogously.
\end{definition}

The products and pairings are natural in the following sense.

\begin{proposition}[\cite{costenoble}, Theorems 1.6.15 and 1.6.17] 
 Let $X, Y, Z$ and $X', Y', Z'$ be $G$-spaces and $g:X\to X'$, $h:Y\to Y'$, $f:Z\to Z'$ be $G$-maps. Let $d:Z\to X\smash Y$ and $d':Z'\to X'\smash Y'$ be $G$-maps.
 Then for Mackey functors $\conj{U}, \conj{T}, \jnoc{S}$ and $x'\in H^*_G(X';\conj{T}), y'\in H^*_G(Y';\conj{U})$ we have
 \[
  g^*(x')\times h^*(y')=(g\smash h)^*(x'\times y').
 \]
 Consequently, if the diagram
 \[
  \xymatrix{
           Z\ar[r]^d\ar[d]_f&X\smash Y\ar[d]^{g\smash h}\\
           Z'\ar[r]^{d'}&X'\smash Y'
           }
 \]
 commutes, then
 \[
  g^*(x')\cup_dh^*(y')=f^*(x'\cup_{d'} y').
 \]
 Furthermore, if $z\in H_*^G(Z;\jnoc{S})$, then
 \[
  \scalar{y', f_*(z)}_{d'}=g_*\scalar{h^*(y'), z}_d.
 \]
 Similar formulas hold for the products and pairings involving the dual theories.

 In addition, the products and pairings respect restriction to subgroups and fixed spaces.
\end{proposition}

If the map $d$ is a diagonal $d:X\to X\smash X$, or closely related to such a diagonal, we also use the notation $\xi\cap_dx$ for $\scalar{\xi,x}_d$ and speak of the 
cap product. 

For the rest of the paper, we use the following convention. Since we will mostly be interested in the coefficient systems $\conj{A}_{\quot GG}$ and $\jnoc{A}^{\quot GG}$, we will write
\[
 H^G_\alpha(X)=H^G_\alpha(X;\jnoc{A}^{\quot GG}),\;H_G^\alpha(X)=H_G^\alpha(X;\conj{A}_{\quot GG}),
\]
and similarly for the dual theories. We will denote the spectrum $H\conj{A}_{\quot GG}$ representing ordinary cohomology with Burnside coefficients with $\calH$. Along these 
lines, we also suppress the natural isomorphisms $\conj{A}_{\quot GG}\Box\conj{A}_{\quot GG}\cong\conj{A}_{\quot GG}$ and 
$\conj{A}_{\quot GG}\nabla\jnoc{A}^{\quot GG}\cong\jnoc{A}^{\quot GG}$.

Occasionally we will need associativity of the cup product. For simplicity, we only state it for cohomology with Burnside coefficients. There is an obvious more general version. The proof 
is immediate from the definition. 

\begin{proposition}\label{prop:cupassociativity}
 Let $A, B, C, X, Y, Z$ be $G$-spaces and let $d_1:A\to X\smash Y$, $d_2:B\to A\smash Z$, $d_3:C\to Y\smash Z$ and $d_4:B\to X\smash C$ be $G$-maps such that the diagram
 \[
  \xymatrix{
           B\ar[r]^{d_2}\ar[d]^{d_4}&A\smash Z\ar[d]^{d_1\smash\id}\\
           X\smash C\ar[r]^{\id\smash d_3}&X\smash Y\smash Z
           }
 \]
 commutes. Then 
 \[
  \xymatrix{
           H^\alpha_G(X)\tensor H^\beta_G(Y)\tensor H^\gamma_G(Z)\ar[rr]^{\cup_{d_1}\tensor\id}\ar[d]^{\id\tensor\cup_{d_3}}&&H^{\alpha+\beta}_G(A)\tensor H^\gamma_G(Z)\ar[d]^{\cup_{d_2}}\\
           H^\alpha_G(X)\tensor H^{\beta+\gamma}_G(C)\ar[rr]^{\cup_{d_4}}&&H^{\alpha+\beta+\gamma}_G(B)
           }
 \]
 commutes.
\end{proposition}

\section{Orientability and Lefschetz Classes}
We start this chapter by a general overview of equivariant orientability as developed in \cite{may2}.

Let $B$ be a $G$-space. The equivariant fundamental groupoid $\Pi_GB$ is defined as the category with objects $G$-maps $\varphi:\quot GH\to B$, with $H$ any subgroup of 
$G$. A morphism from $\varphi:\quot GH\to B$ to $\psi:\quot GK\to B$ is a pair $([\omega], \alpha)$, where $\alpha:\quot GH\to\quot GK$ is a $G$-map, 
$\omega:\quot GH\times\incc{0,1}\to B$ is a $G$-homotopy between $\varphi$ and $\psi\circ\alpha$, and $[\omega]$ is the $G$-homotopy class of $\omega$ with fixed endpoints.

We can also think of $\varphi$ and $\psi$ as points in $B^H$, $B^K$, respectively, and $[\omega]$ is a homotopy class of paths in $B^H$ connecting $\varphi$ and $\psi\circ\alpha$,
where the latter denotes the element $\psi\circ\alpha([e])\in B^H$.

If $p:E\to B$ is a $G$-vector bundle, the pullback construction induces a functor $p^*:\Pi_GB\to\conj{\calV}_G$, where $\conj{\calV}_G$ is the category of $G$-vector bundles over orbits
and $G$-homotopy classes of bundle maps. Precisely, this is done as follows. On objects, $p^*(\varphi)=\varphi^*(p)$ is the pullback bundle via the morphism $\varphi$.
For a morphism $([\omega], \alpha)$, we have the diagram
\[
 \xymatrix{
          \varphi^*(p)\ar[r]\ar[d]^{i_0}&E\ar[d]^p\\
          \varphi^*(p)\times\incc{0,1}\ar[ru]^{\exists\,\tilde{\omega}}\ar[r]\ar[d]&B\\
          \quot GH\times\incc{0,1}\ar[ru]_\omega&
          }
\]
The horizontal map in the middle is a homotopy defined by commutativity of the lower triangle, and by the fibration property of $G$-vector bundles, there exists an equivariant 
lift $\tilde{\omega}$ of that homotopy. $\tilde{\omega}_1$ is a bundle map covering $\psi\circ\alpha$ and it induces a map $\varphi^*(p)\to(\psi\circ\alpha)^*(p)$.

We can extend our definition of $RO(G)$ to $G$-vector bundles over orbits. We choose a set of representatives $W_i$, $i\in I$, of $G$-isomorphism classes of irreducible 
$H$-representations for every subgroup $H$ of $G$. Then we fix an isomorphism from any bundle over $\quot GH$ to a bundle of the form 
\[
G\times_H(W_{i_1}^{k_1}\oplus\dots\oplus W_{i_n}^{k_n}).
\]
We can furthermore assume that the underlying spaces of the $W_i$ are just real Euclidean spaces, and that the above isomorphism is the identity when the bundle is a bundle
of the given form. This specifies a functor from the category $\conj{\calV}_G$ to the category $\calV_G$ of bundles of the form as above. The pullback functor $p^*$ obviously 
extends to a functor into $\calV_G$.

\begin{definition}[\cite{may2}, Definition 2.8]
 Let $p:E\to B$ be a $G$-vector bundle. $p$ is said to be $G$-orientable, if the pullback functor $p^*:\Pi_GB\to\calV_G$, whenever $([\omega], \alpha)$, $([\omega'], \alpha')$ are two
 morphisms from $\varphi$ to $\psi$ and $\alpha=\alpha'$, satisfies that $p^*([\omega], \alpha)=p^*([\omega'], \alpha)$, i.e. $p^*$ is independent of the path class $[\omega]$ for fixed 
 $\alpha$.

 A $G$-manifold $M$ is called $G$-orientable, if its tangential bundle is $G$-orientable.
\end{definition}

To make use of the equivariant duality theory of chapter one in \cite{costenoble}, we restrict our attention to a special case of $G$-manifolds, an assumption which we will later
drop. It may be interesting to see whether one obtains the same invariants by using the twisted theory of chapter four of \cite{costenoble}, which establishes duality for 
arbitrary $G$-manifolds.

The bundles and manifolds we will be looking at are defined as follows.

\begin{definition}
 Let $p:E\to B$ be a $G$-vector bundle. $p$ is said to be a $V$-vector bundle, if every fibre $p^{-1}(x)$ over a point $x$ is $G_x$-isomorphic to the restriction of a 
 $G$-representation $V$.

 A $G$-manifold $M$ is said to be a $V$-manifold, if its tangential bundle is a $V$-vector bundle.
\end{definition}

An alternative and more conceptual description of $V$-vector bundles is that a $V$-vector bundle is a map $p:E\to B$ such that every point $x\in B$ has a neighbourhood $U$ such 
that $p^{-1}(U)$ is $G_x$-bundle isomorphic to $U\times V$. A $V$-manifold then is a $G$-manifold $M$ such that every point $x\in M$ is contained in a chart $G_x$-isomorphic to 
$V$.

The notions of Thom classes and fundamental classes for $V$-bundles and $V$-manifolds are straightforward generalizations of the non-equivariant situation.

\begin{definition} Let $p:E\to B$ be a $V$-vector bundle and $M$ be a $V$-manifold.
 \begin{enumerate}[i)]
 \item Let $T(p)$ be the Thom space of $p$. A Thom class for $p$ is a class $\tau\in H^V_G(T(p))$ such that for every $G$-map 
 $\varphi:\quot GH\to B$, the image of $\tau$ under the maps
 \begin{eqnarray*}
  H^V_G(T(p))&\to&H^V_G(T(\varphi^*(p)))\\
             &\cong&H^V_G(G_+\wedge_H\setS^V)\\
             &\cong&H^V_H(\setS^V)\\
             &\cong&A(H)
 \end{eqnarray*}
 is a generator (i.e. a unit).

 \item If $M$ is compact, a fundamental class for $M$ is a class $\calO\in\calH_V^G(M_+)$ such that the following holds. For any point $x\in M$, let $\nu_x$ be a normal bundle
 of the embedding $Gx\to M$ and let $T(\nu_x)$ be the Thom space of this embedding. Let $L_x$ be the tangential space of $Gx$ at $x$ and $\psi:M_+\to T(\nu_x)$ be the 
 Pontryagin-Thom map. Then the image of $\calO$ under the maps
 \begin{eqnarray*}
  \calH_V^G(M_+)&\slra{\psi_*}&\calH_V^G(T(\nu_x))\\
                &\cong&\calH_V^G(G_+\wedge_H\setS^{V-L_x})\\
                &\cong&\calH_V^H(\setS^V)\\
                &\cong&A(H)
 \end{eqnarray*}
 is a generator.

 \item If $M$ is compact with boundary, a fundamental class for $M$ is a class $\calO\in\calH_V^G(\quot M{\partial M})$ such that for any point $x\in M\setminus\partial M$ with
 normal bundle $\nu_x$, Thom space $T(\nu_x)$, $L_x$ and $\psi:\quot M{\partial M}\to T(\nu_x)$ as above, the image of $\calO$ in
 \begin{eqnarray*}
  \calH_V^G(\quot M{\partial M})&\slra{\psi_*}&\calH_V^G(T(\nu_x))\\
                &\cong&\calH_V^G(G_+\wedge_H\setS^{V-L_x})\\
                &\cong&\calH_V^H(\setS^V)\\
                &\cong&A(H)
 \end{eqnarray*}
 is a generator.

 \item If $N\subseteq M$ is an orientable compact $W$-submanifold of the compact $V$-manifold $M$, we define the fundamental class $\calO^M_N$ of $N$ in $M$ to be the image 
 ${i^M_N}_*(\calO_N)\in\calH_W^G(M_+)$ of the fundamental class $\calO_N\in\calH_W^G(N_+)$ under the inclusion $i^M_N:N\to M$.
 \end{enumerate}
\end{definition}

We prove existence of fundamental classes by using the equivariant Thom isomorphism theorem of \cite{costenoble}. More precisely, we will show that a fundamental class exists if 
and only if a Thom class for the normal bundle of an embedding into a $G$-representation exists. Here is the Thom isomorphism result.

\begin{theorem}[\cite{costenoble}, Theorem 1.7.5]\label{thm:thomiso}
 Let $p:E\to B$ be a $V$-vector bundle. $p$ is $G$-orientable if and only if it has an equivariant Thom class $\tau\in H^V_G(T(p))$. In that case, the map
\[
 H^\alpha_G(B)\to H^{\alpha+V}_G(T(p)),\;a\mapsto a\cup_d\tau
\]
is an isomorphism, where $d:T(p)\to B_+\wedge T(p)$ is the Thom diagonal.
\end{theorem}

The connection between Thom classes and fundamental classes is laid bare with the next proposition, for which we need another result first, relating pullback bundles and normal
bundles of embeddings.
\begin{lemma}
 Let $M$ be a $V$-manifold and $\varphi:\quot GH\to M$ be a $G$-embedding. Let $M\to W$ be a $G$-embedding of $M$ into an orthogonal $G$-representation $W$ with normal bundle 
 $\nu$. Furthermore, let $\nu_\varphi$ be the normal bundle of the embedding $\varphi$. Then the Spanier-Whitehead dual of $T\nu_\varphi$ is $\Sigma^{-W}T(\varphi^*\nu)$.
\end{lemma}

\begin{proof}
We have $T\nu_\varphi\cong G_+\smash_H\setS^N$, where $N$ is a complement to $L=T_{[e]}\quot GH$ in $V$. Furthermore, 
$T(\varphi^*\nu)\cong G_+\smash_H\setS^{W-V}\cong{\quot GH}_+\smash\setS^{W-V}$. Using the fact that $D(\Sigma^QX)=\Sigma^{-Q}D(X)$ for any $G$-representation $Q$ and again the 
calculation of the dual of an orbit $\quot GH$ to be $G_+\smash_H\setS^{-L}$, we see that
\begin{eqnarray*}
 D(T(\varphi^*\nu))&\cong& D({\quot GH}_+\smash\setS^{W-V})\\
                   &\cong&\Sigma^{V-W}D({\quot GH}_+)\\
                   &\cong&\Sigma^{V-W}G_+\smash_H\setS^{-L}\\
                   &\cong&\Sigma^{-W}G_+\smash_H\setS^N\\
                   &\cong&\Sigma^{-W}T\nu_\varphi.
\end{eqnarray*}
\end{proof}

Now we can establish in detail the promised connection between Thom classes and fundamental classes, which has been sketched in \cite{costenoble}. Here and in the following, we 
will assume that if $M$ is a manifold with boundary embedded in a $G$-representation $W$, then $W$ is of the form $W=W'\oplus\setR$ and $\partial M$ embeds into $W'\times\{0\}$,
$M\setminus\partial M$ embeds into $W'\times\inoo{0,\infty}$.

\begin{proposition}\label{prop:fundamentalthomclass}
Let $M$ be a compact $V$-manifold with (possibly empty) boundary and let $M\to W$ be an embedding into an orthogonal $G$-representation $W$. Let $\nu$ be the normal bundle of 
that embedding. The following sequence consists of isomorphisms
\begin{eqnarray*}
 H^{\alpha+W-V}_G(T\nu)&\cong&\{T\nu,\Sigma^{\alpha+W-V}\calH\}_G\\
                       &\cong&\{\Sigma^WD(\quot M{\partial M}),\Sigma^{\alpha+W-V}\calH\}_G\\
                       &\cong&\{\setS^{V-\alpha},\quot M{\partial M}\wedge\calH\}_G\\
                       &\cong&\calH_{V-\alpha}^G(\quot M{\partial M}).
\end{eqnarray*}
In addition, taking $\alpha=0$, a class $\tau\in H^{W-V}_G(T\nu)$ is a Thom class for $\nu$ if and only if its image $\calO\in\calH_V^G(\quot M{\partial M})$ is a fundamental 
class for $M$.
\end{proposition}

\begin{proof}
 The basic ingredient in the proof is the calculation of the Spanier-Whitehead dual of $\quot M{\partial M}$. This is calculated in \cite{may}, III.5.4, as $\Sigma^{-W}T\nu$. Then
 in our sequence, the first map is an isomorphism by definition of the spectrum $\calH$, the second map is Spanier-Whitehead duality. The third map is suspension and the last map 
 again is just definition of dual homology, so we indeed have an isomorphism and the image of a Thom class is a strong candidate for a fundamental class. So let 
 $\varphi:\quot GH\to M\setminus\partial M$ be a $G$-embedding, $\nu_\varphi$ be a normal bundle of the embedding and $\psi:\quot M{\partial M}\to T\nu_\varphi$ the 
 associated Pontryagin-Thom map. We consider the diagram
\[
 \xymatrix{
 H^{W-V}_G(T\nu)\ar[rr]^{(T\varphi^*)^*}\ar[d]^\cong                                  &&  H^{W-V}_G(T(\varphi^*\nu))\ar[d]^\cong\\
 \{\Sigma^VT\nu, \Sigma^W\calH\}_G\ar[rr]^{\circ\Sigma^V(T\varphi^*)}\ar[d]^\cong           &&  \{\Sigma^VT(\varphi^*\nu), \Sigma^W\calH\}_G\ar[d]^\cong\\
 \{\Sigma^V\Sigma^WD(\quot M{\partial M}), \Sigma^W\calH\}_G\ar[d]^\cong\ar[rr]^{\circ\Sigma^{V+W}D(\psi)}  &&  \{\Sigma^V\Sigma^WD(T\nu_\varphi), \Sigma^W\calH\}_G\ar[d]^\cong\\
 \{\Sigma^VD(\quot M{\partial M}),\calH\}_G\ar[rr]^{\circ\Sigma^VD(\psi)}\ar[d]^\cong                      &&  \{\Sigma^VD(T\nu_\varphi),\calH\}_G\ar[d]^\cong\\
 \{\setS^V, \quot M{\partial M}\wedge\calH\}_G\ar[rr]^{(\psi\wedge\id_\calH)\circ}\ar[d]^\cong                 &&  \{\setS^V, T\nu_\varphi\wedge\calH\}_G\ar[d]^\cong\\
 \calH_V^G(\quot M{\partial M})\ar[rr]^{\psi_*}                                              &&  \calH_V^G(T\nu_\varphi).
 }
\]
 Commutativity is seen as follows. The first square is just a definition, the second square is the natural duality identification of Thom spaces, together with the fact, see 
 \cite{may} III.5.5, that the Pontryagin-Thom map (more precisely its suspension) is dual to the inclusion. The third square is a suspension isomorphism, the fourth square is 
 Spanier-Whitehead duality and the last square again a definition. Since a Thom class in the upper left maps to a generator in the upper right, the same is true for the image of 
 the Thom class in the lower left, which therefore must be a fundamental class.
\end{proof}

Using the preceeding proposition together with the Thom isomorphism theorem and some more calculations of duals from \cite{may} III.5, we obtain the following Poincar\'{e}
duality theorem. It appeared in \cite{costenoble}, but without a proof, and is a special case of the more general discussion in \cite{may} III.6.

\begin{theorem}[\cite{costenoble}, Theorem 1.7.10]
 Let $M$ be a compact orientable $V$-manifold with possibly empty boundary. Then $M$ has a fundamental class, and the following pairings induce isomorphisms, the equivariant
 Poincar\'{e} duality isomorphisms, given by capping with the fundamental class of $M$.
 \begin{enumerate}[i)]
  \item With $d:\quot M{\partial M}\to M_+\smash\left(\quot M{\partial M}\right)$ induced by the diagonal of $M_+$,
  \[
   H^\alpha_G(M_+)\tensor\calH^G_V(\quot M{\partial M})\slra{\cap_d}\calH_{V-\alpha}^G(\quot M{\partial M}).
  \]

  \item With $d:\quot M{\partial M}\to\left(\quot M{\partial M}\right)\smash M_+$ induced by the diagonal of $M_+$,
  \[
   H^\alpha_G(\quot M{\partial M})\tensor\calH_V^G(\quot M{\partial M})\slra{\cap_d}\calH^G_{V-\alpha}(M_+).
  \]
 \end{enumerate}
 
\end{theorem}
\begin{proof}
 The existence of a fundamental class follows immediately from Theorem \ref{thm:thomiso} and Proposition \ref{prop:fundamentalthomclass}.

 For Poincar\'{e} duality, we embed $M$ into a $G$-representation $W$ with the same convention as before Proposition \ref{prop:fundamentalthomclass} if $M$ has non-empty boundary.
 In particular, in that case $W=W'\oplus\setR$ and $\partial M$ embeds into $W'\times\{0\}$.
 
 Let $T\nu$ be the Thom space of the normal bundle of the embedding $M\to W$, $T\nu'$ be the Thom space of the normal bundle of the embedding $\partial M\to W'\times\{0\}$. 

 To establish i), we use that the dual of $\quot M{\partial M}$ is $\Sigma^{-W}T\nu$ (Theorem III.5.4 of \cite{may}). Then we have the commutative diagram
 \[
  \xymatrix{
           \{M_+,\setS^\alpha\smash\calH\}_G\tensor\{\setS^V,\quot M{\partial M}\smash\calH\}_G\ar[r]^{\quad\quad\quad\quad\cap_d}\ar[d]^\cong&\{\setS^{V-\alpha},\quot M{\partial M}\smash\calH\}_G\\
           \{M_+,\setS^\alpha\smash\calH\}_G\tensor\{T\nu, \setS^{W-V}\smash\calH\}_G\ar[r]^{\quad\quad\quad\quad\cup}&\{T\nu,\setS^{\alpha+W-V}\smash\calH\}_G\ar[u]_\cong\\
           }
 \]
 where the vertical arrows are given by Proposition \ref{prop:fundamentalthomclass} and the cup product at the bottom is with respect to the Thom diagonal $T\nu\to M_+\smash T\nu$.

 Since again by Proposition \ref{prop:fundamentalthomclass} a fundamental class in $\calH_V^G(\quot M{\partial M})$ maps to a Thom class in $H^{W-V}_G(T\nu)$, the claim
 follows from the Thom isomorphism theorem \ref{thm:thomiso}.

 To establish ii), we use that the dual of $M_+$ is $\Sigma ^{-W}\quot{T\nu}{T\nu'}$ (again from Theorem III.5.4 of \cite{may}). We have the commutative diagram
 \[
  \xymatrix{
           \{\quot M{\partial M},\setS^\alpha\smash\calH\}_G\tensor\{\setS^V,\quot M{\partial M}\smash\calH\}_G\ar[r]^{\quad\quad\quad\quad\cap_d}\ar[d]^\cong&\{\setS^{V-\alpha},M_+\smash\calH\}_G\\
           \{\quot M{\partial M},\setS^\alpha\smash\calH\}_G\tensor\{T\nu, \setS^{W-V}\smash\calH\}_G\ar[r]^{\quad\quad\quad\quad\cup}&\{\quot{T\nu}{T\nu'},\setS^{\alpha+W-V}\smash\calH\}_G\ar[u]_\cong\\
           }
 \]
 where the vertical maps are again given by Proposition \ref{prop:fundamentalthomclass} (or the obvious modification thereof for the map on the right), and the cup product is 
 induced by the Thom diagonal
 \[
  \quot{T\nu}{T\nu'}\to\quot M{\partial M}\smash T\nu.
 \]
 Again, the claim follows from the Thom isomorphism theorem.
\end{proof}

Using these isomorphisms, we can define the intersection product to be the map dual to the cup product. We only give an explicit definition in the case of the cup product on
cohomology, for $M$ an orientable compact $V$-manifold, thus getting an intersection product on dual homology. The other cases are treated similarly, but we will not need them in what is to follow.

\begin{definition}
 Let $M$ be an orientable compact $V$-manifold. Let $P:H^*_G(M_+)\to\calH_{V-*}^G(M_+)$ be the Poincar\'{e} duality isomorphism. The intersection product on $M$ is defined as
\[
 \bullet:\calH_\alpha^G(M_+)\tensor\calH_\beta^G(M_+)\to\calH_{\alpha+\beta-V}^G(M_+),\;x\tensor y\mapsto P(P^{-1}(y)\cup P^{-1}(x)).
\]
\end{definition}

In the following discussion, we assume that $M$ is a compact $V$-manifold without boundary. 

In the non-equivariant case it follows easily from the Kuenneth theorem that the homology cross product of two fundamental classes is a fundamental class for the product manifold.
The situation is more difficult in the equivariant case and we take a more direct approach to prove this fact. The essential ingredient is to show that the cohomology cross 
product of two Thom classes is a Thom class for the product bundle.

\begin{proposition}\label{prop:thomclassproduct}
 Let $p:E\to B$ be an orientable $V$-vector bundle, $q:E'\to B'$ be an orientable $W$-vector bundle, with Thom classes $\tau\in H^V_G(T(p))$, $\tau'\in H^W_G(T(q))$, respectively.
 Then the product bundle $p\times q:E\times E'\to B\times B'$ is an orientable $V\oplus W$-bundle and a Thom class is given by the image of $\tau\tensor\tau'$ under 
 \[
  H^V_G(T(p))\tensor H^W_G(T(q))\slra{\times}H^{V+W}_G(T(p)\wedge T(q))\cong H^{V+W}_G(T(p\times q)).
 \]
\end{proposition}

\begin{proof}
We show that the stated class is a Thom class for $p\times q$. So let $\varphi=(\varphi_1,\varphi_2):\quot GH\to B\times B'$ be a $G$-map. We build a commutative diagram as 
follows. The diagram
\begin{footnotesize}
\[
 \xymatrix{
  H^V_G(T(p))\tensor H^W_G(T(q))\ar[r]^\times\ar[d] & H^{V+W}_G(T(p)\wedge T(q))\ar[r]^{\quad\quad\quad\cong}\ar[d] & H^{V+W}_G(T(p\times q))\ar[d]\\
  H^V_G(T(\varphi_1^*p))\tensor H^W_G(T(\varphi_2^*q))\ar[r]^\times & H^{V+W}_G(T(\varphi_1^*p)\wedge T(\varphi_2^*q))\ar[r]^{\quad\quad\cong} & H^{V+W}_G(T(\varphi_1^*p\times\varphi_2^*q))\\
  }
\]
\end{footnotesize}
commutes, the left square due to naturality of the cross product, the right square due to naturality of the multiplicativity of Thom spaces. The horizontal maps are actually cup
products with respect to the isomorphisms $T(p)\smash T(q)\cong T(p\times q)$ and $T(\varphi_1^*p))\smash T(\varphi_2^*q)\cong T(\varphi_1^*p\times\varphi_2^*q)$.
So we have
\[
 \xymatrix{
  H^V_G(T(p))\tensor H^W_G(T(q))\ar[r]^\cup\ar[d]&H^{V+W}_G(T(p\times q))\ar[d]\\
  H^V_G(T(\varphi_1^*p))\tensor H^W_G(T(\varphi_2^*q))\ar[r]^\cup & H^{V+W}_G(T(\varphi_1^*p\times\varphi_2^*q))
  }
\]
The same reasoning gives a commutative diagram
\[
 \xymatrix{
  H^V_G(T(\varphi_1^*p))\tensor H^W_G(T(\varphi_2^*q))\ar[r]^\cup\ar[d]^\cong & H^{V+W}_G(T(\varphi_1^*p\times\varphi_2^*q))\ar[d]^\cong\\
  H^V_G(\Sigma^V{\quot GH}_+)\tensor H^W_G(\Sigma^W{\quot GH}_+)\ar[r]^\cup & H^{V+W}_G(\Sigma^{V+W}{\quot GH}_+\smash{\quot GH}_+)
          }
\]
We can paste these two diagrams together and obtain
\[
 \xymatrix{
  H^V_G(T(p))\tensor H^W_G(T(q))\ar[r]^\cup\ar[d]&H^{V+W}_G(T(p\times q))\ar[d]\\
  H^V_G(T(\varphi_1^*p))\tensor H^W_G(T(\varphi_2^*q))\ar[r]^\cup\ar[d]^\cong & H^{V+W}_G(T(\varphi_1^*p\times\varphi_2^*q))\ar[d]^\cong\\
  H^V_G(\Sigma^V{\quot GH}_+)\tensor H^W_G(\Sigma^W{\quot GH}_+)\ar[r]^\cup\ar[d]^\cong & H^{V+W}_G(\Sigma^{V+W}{\quot GH}_+\smash{\quot GH}_+)\ar[d]^\cong\\
  H^0_G({\quot GH}_+)\tensor H^0_G({\quot GH}_+)\ar[r]^\times\ar[dd]^\cong & H^0_G({\quot GH}_+\smash{\quot GH}_+)\ar[d]^{\Delta^*}\\
  & H^0_G({\quot GH}_+)\ar[d]^\cong\\
  A(H)\tensor A(H)\ar[r]^\mu&A(H).
 }
\]
We have established commutativity of the top two squares. The third square commutes due to naturality of the product with respect to suspension, the last square is a restatement 
of the fact that the cup product generalizes the ring multiplication in $A(H)$. This follows directly from the definition of the product and the fact that the Burnside ring 
product is induced by smash product of representants. 

If we start with the product of two Thom classes in the upper left, these are mapped to a tensor product of units in the lower left by definition of the Thom class. So they are mapped
to a unit in the lower right.

But the map down the right column is just the composition 
\begin{eqnarray*}
H^{V+W}_G(T(p\times q))&\to& H^{V+W}_G(T(\varphi^*(p\times q))\\
                       &\cong& H^{V+W}_G(\Sigma^{V+W}{\quot GH}_+)\\
                       &\cong& H^0_G({\quot GH}_+)\\
                       &\cong& A(H),
\end{eqnarray*}
since $\varphi_1\times\varphi_2\circ\Delta=(\varphi_1,\varphi_2)=\varphi$. Thus, the potential Thom class maps to a generator under restriction via $\varphi$, which characterizes it as a 
Thom class.
\end{proof}

Our intended result for fundamental classes follows immediately.

\begin{corollary}
 Let $M$ be an orientable $V$-manifold. Then $M\times M$ is an orientable $V\oplus V$-manifold and the fundamental class of $M\times M$ can be chosen such that in the 
 commutative diagram
\[
 \xymatrix{
          H^{W-V}_G(T\nu)\tensor H^{W-V}_G(T\nu)\ar[r]^\times\ar[d]^\cong& H^{2W-2V}_G(T(\nu\times\nu))\ar[d]^\cong\\
          H^0_G(M_+)\tensor H^0_G(M_+)\ar[r]^\times\ar[d]^{P\tensor P}&H^0_G((M\times M)_+)\ar[d]^{P_\times}\\
          \calH^G_V(M_+)\tensor\calH^G_V(M_+)&\calH^G_{2V}((M\times M)_+),
          }
\]
the element $\calO\tensor\calO$ in the lower left maps to $\tau\tensor\tau$ in the upper left, where $\tau$ is a Thom class for the normal bundle $\nu$ of the embedding $M\to W$, 
$\tau\tensor\tau$ maps to a Thom class of the normal bundle of the embedding $M\times M\to W\times W$, and this Thom class maps to the fundamental class of $M\times M$ in the
lower right.
\end{corollary}

\begin{proof}
 We just define the fundamental class of $M\times M$ via the given mapping property. Proposition 4.5 shows that this indeed gives a fundamental class.
\end{proof}

This characterization of fundamental classes in a product manifold allows us to define a Lefschetz number in the case of compact $V$-manifolds. The general case is then a 
standard modification.

\begin{definition}
 Let $M$ be a compact $V$-manifold and $f:M\to M$ a $G$-map. Let $\calO_\Delta\in\calH_V^G((M\times M)_+)$ be the fundamental class of the diagonal submanifold of 
 $M\times M$ in $M\times M$, and let $\calO_\Gamma\in\calH_V^G((M\times M)_+)$ be the fundamental class of the graph manifold $\Gamma=\{(x, f(x))\;|\;x\in M\}\subseteq M\times M$. The 
 equivariant Lefschetz number of $f$ is defined to be the intersection product of $\calO_\Delta$ and $\calO_\Gamma$, followed by the evaluation map 
 $\eps:\calH_0^G((M\times M)_+)\to\calH_0^G(\setS^0)\cong A(G)$:
  \[
   L_G(f)=\eps(\calO_\Delta\bullet\calO_\Gamma).
  \] 
\end{definition}

To generalize this definition to a more general type of manifold, we first rephrase it a bit. Consider the diagram
\[
 \xymatrix{
          M_+\ar[rr]^\Delta\ar[d]^{(\id,f)}&&(M\times M)_+\ar[d]^{((\id,f), (\id,f))}\\
          (M\times M)_+\ar[rr]^{\Delta_\times}&&(M\times M\times M\times M)_+
          },
\]
where $\Delta_\times$ is the diagonal embedding of the product $M\times M$. Using naturality of the cap product, we can conclude that for classes 
$x\in H^\alpha_G((M\times M)_+)$ and $y\in\calH^G_\beta(M_+)$, we have
\[
 \Delta_*((\id,f)^*(x)\cap y)=x\cap(\id,f)_*(y).
\]
We apply this to the intersection product $\calO_\Delta\bullet\calO_\Gamma$, and use the fact that the dual of $\calO_\Delta$ is given by the element 
$\tau^{M\times M}_\Delta=\psi^*(\tau)$, where $\tau$ is the Thom class of the embedding $\Delta\subseteq M\times M$ and $\psi:(M\times M)_+\to T\nu^{M\times M}_\Delta$ is the 
associated Pontryagin-Thom map. Similarly, the dual of $\calO_\Gamma$ is an element $\tau^{M\times M}_\Gamma$, associated to the graph embedding $\Gamma\to M\times M$. This 
duality statement will be proven in Lemma \ref{lem:fundamentalthomduality}. We obtain
\begin{eqnarray*}
 \calO_\Delta\bullet\calO_\Gamma&=&(\tau^{M\times M}_\Delta\cup\tau^{M\times M}_\Gamma)\cap\calO_{M\times M}\\
                                &=&\tau^{M\times M}_\Delta\cap(\tau^{M\times M}_\Gamma\cap\calO_{M\times M})\\
                                &=&\tau^{M\times M}_\Delta\cap(\id,f)_*(\calO_M)\\
                                &=&(\id,f)_*((\id,f)^*(\tau^{M\times M}_\Delta)\cap\calO_M)).
\end{eqnarray*}
Since the diagram
\[
 \xymatrix{
         M\ar[r]^{(\id,f)}\ar[rd]&M\times M\ar[d]\\
           & {*}
          }
\]
commutes, it follows for the equivariant Lefschetz number that
\[
 L_G(f)=\eps((\id,f)^*(\tau^{M\times M}_\Delta)\cap\calO_M)).
\]
Writing this in diagrammatic form, we have that the equivariant Lefschetz number equals the image of the Thom class $\tau$ of the diagonal embedding of $M$ into $M\times M$ under the 
sequence of maps
\[
 H^V_G(T\nu^{M\times M}_\Delta)\slra{\psi^*}H^V_G((M\times M)_+)\slra{(\id,f)^*}H^V_G(M_+)\slra{\cap\calO_M}\calH_0^G(M_+)\slra{\eps}A(G),
\]
where the first map is induced by the Pontryagin-Thom map. 

 Now if $M$ is a non-compact orientable $V$-manifold and $f$ has compact fixed point set, let $N$ be an invariant open neighbourhood of $Fix(f)$, such that $\conj{N}$ is an 
 invariant manifold with boundary. As an open subset of an orientable $V$-manifold, $N$ is an orientable $V$-manifold. Let $T\nu^{M\times M}_\Delta$ be the Thom space of the 
 diagonal embedding of $M$ into $M\times M$. We can realize this Thom space as the compactification of an invariant neighbourhood of $\Delta$ such that the image of the boundary 
 of $N$ under $(\id, f)$ does not lie in that neighbourhood. Then $(\id, f)$ induces a map $(\id, f):\quot{\conj{N}}{\partial N}\to T\nu^{M\times M}_\Delta$, and we define the 
 equivariant Lefschetz number of $f$ to be the image of the Thom class under
 \[
  H^V_G(T\nu^{M\times M}_\Delta)\slra{(\id,f)^*} H^V_G(\quot{\conj{N}}{\partial N})\slra{\cong}\calH_0^G(N_+)\slra{\eps}A(G)
 \]
 In the general case where $M$ is not necessarily orientable or a $V$-manifold, we embed $M$ into a $G$-representation $V$ with invariant tubular neighbourhood $U$. As an open 
 subset of $V$, $U$ is a non-compact orientable $V$-manifold. We define a map $f_0:U\to U$ by $f_0=i\circ f\circ r$, where $r:U\to M$ is the tubular retraction, $i:M\to U$ the 
 embedding. If the fixed point set of $f_0$ (equal to that of $f$) is compact, we define
 \[
  L_G(f)=L_G(f_0),
 \]
 the right hand term having been defined previously. Summarizing, we have defined an equivariant Lefschetz number for $G$-manifolds $M$ with finite orbit type (those are 
 embeddable into finite dimensional $G$-representations, see \cite{tomdieck}), and $G$-maps $f:M\to M$ with compact fixed point set. 

 We should address the question of independence of the equivariant Lefschetz number of all the choices made in the definition. This is straightforward and does not differ from the 
 non-equivariant situation, which can be found in \cite{nussbaum}, so we omit a rigorous proof. The general idea is that, given two embeddings of $M$ into $V$ and $W$, $M$ embeds
 diagonally into $V\oplus W$, and the Lefschetz numbers defined via the embedding into $V$ and into $W$ both equal the Lefschetz number defined via the diagonal embedding.

\section{Restrictions and submanifolds}
In order to derive properties of the equivariant Lefschetz number, we need to know what happens when we restrict to fixed points or subgroups. In the following, we again restrict
ourselves to the case where $M$ is a compact $V$-manifold.

Thom classes and fundamental classes behave well under restriction to subgroups and fixed sets. The following result was proven for Thom classes in \cite{costenoble}, the proofs
are very similar. With the obvious adaptions, it is also true for manifolds with boundary. 

\begin{proposition}\label{prop:fundamentalrestriction}
 Let $M$ be a $V$-manifold. Then the following are equivalent for a class $\calO\in\calH_V^G(M_+)$:
 \begin{enumerate}[i)]
  \item $\calO$ is a fundamental class for $M$.

  \item $\calO\big|H$ is a fundamental class for $M$ as an $H$-manifold for every subgroup $H$ of $G$.

  \item $\calO^H$ is a fundamental class for $M^H$ as a $W(H)$-manifold for every subgroup $H$ of $G$.

  \item $\calO^H\big|e$ is a fundamental class for $M^H$ as an $e$-manifold for every subgroup $H$ of $G$.
 \end{enumerate}
\end{proposition}

\begin{proof}
 i) $\implies$ ii): A fundamental class $\calO$ is characterized by the property that for every $G$-embedding $\varphi:\quot GH\to M$, $\calO$ maps to a generator 
 of $\calH_V^G(T\nu_\varphi)$, where $\nu_\varphi$ is the Thom space of a normal bundle associated with the embedding. 

 So let $\varphi:\quot HK\to M$ be an $H$-map and let $\Phi:\quot GK\to M$ be its unique extension to a $G$-map. We can write
 \[
  V=L\oplus N=L_1\oplus L_2\oplus N,
 \]
 where $L=T_{[e]}\quot GK$, $L_1=T_{[e]}\quot HK$ and $L_2=T_{[e]}\quot GH$. Then, the Thom space of the embedding $\varphi$ is $H$-homeomorphic to 
 $H_+\wedge_K\setS^{L_2\oplus N}$ and the Thom space of the embedding $\Phi$ is $G$-homeomorphic to $G_+\wedge_K\setS^N$. We have to investigate the diagram
 \[
  \xymatrix{\calH_V^G(M_+)\ar[r]^{\psi_*^G}\ar[ddd]^{r^G_H}&\calH_V^G(T\nu_{\Phi})\ar[d]^\cong\\
  &\calH_V^G(G_+\smash_HH_+\smash_K\setS^{V-L})\ar[d]^{w^G_H}\\
  &\calH_V^H(H_+\smash_K\setS^{V-L_1})\ar[d]^\cong\\
  \calH_V^H(M_+)\ar[r]^{\psi^H_*}&\calH_V^H(T\nu_\varphi),
  }
 \]
 where $\psi^G$ and $\psi^H$ are the Pontryagin-Thom maps of the embeddings $\Phi$ and $\varphi$, respectively. The map down the right is essentially the Wirthm\"{u}ller 
 isomorphism. If we could show that the diagram commutes, the claim would follow. For this, it suffices to show that the dual of the diagram commutes. Again we embed $M$ into
 a $G$-representation $W$ with normal bundle $\nu$ and use the facts that
 \[
  D(M_+)\cong\Sigma^{-W}T\nu,\;D(T\nu_\varphi)\cong\Sigma^{-W}T(\varphi^*\nu),\;D(T\nu_\Phi)\cong\Sigma^{-W}T(\Phi^*\nu)
 \]
 and $D(\psi^G)=\Sigma^{-W}T\Phi,\;D(\psi^H)=\Sigma^{-W}T\varphi$. As before, $\nu_\Phi, \nu_\phi$ are normal bundles of the embeddings $\varphi$ and $\Psi$ into $M$. The first 
 three isomorphisms have been used before, and the identification of the dual of the Pontryagin-Thom maps follows from Proposition III.5.5 of \cite{may}.

 The dual diagram therefore has the form
\[
 \xymatrix{
 H^{W-V}_G(T\nu)\ar[r]^{T\Phi^*}\ar[ddd]^{r^G_H}&H^{W-V}_G(T(\Phi^*\nu))\ar[d]^\cong\\
 &H^{W-V}_G(G_+\smash_HH_+\smash_K\setS^{W-V})\ar[d]^{w^G_H}\\
 &H^{W-V}_H(H_+\smash_K\setS^{W-V})\ar[d]^\cong\\
 H^{W-V}_H(T\nu)\ar[r]^{T\varphi^*}&H^{W-V}_H(T(\varphi^*\nu)).
 }
\]
But now it is obvious that this diagram commutes, because the restriction $r^G_H$ is actual restriction of $G$-maps to $H$-maps, $w^G_H$ is the usual adjunction and 
$T\Phi=G_+\smash_HT\varphi$.  

ii) $\implies$ i): This is trivial.

ii) $\implies$ iii): We can assume that $H$ is normal in $G$, otherwise we restrict to $N(H)$ first. Let $K\subseteq G$ be a subgroup containing $H$ and 
$\varphi:\quot GK\to M$ an embedding. If $G_+\wedge_K\setS^N$ is the Thom space of that embedding, since $H$ is normal, we have
\[
 \left(G_+\wedge_K\setS^N\right)^H\cong{\quot GH}_+\wedge_{\quot KH}\setS^{N^H}.
\]
Now consider the commutative diagram
\[
 \xymatrix{
          \calH_V^G(M_+)\ar[r]\ar[d]_\psi&\calH_{V^H}^{\quot GH}(M^H_+)\ar[d]^{\psi^H}\\
          \calH_V^G(G_+\wedge_K\setS^N)\ar[r]\ar[d]_{w^G_K}&\calH_{V^H}^{\quot GH}({\quot GH}_+\wedge_{\quot KH}\setS^{N^H})\ar[d]^{w^{\quot GH}_{\quot GK}}\\
          \calH_V^K(\setS^V)\ar[r]\ar[d]_\sigma&\calH_{V^H}^{\quot KH}(\setS^{V^H})\ar[d]^\sigma\\
          \calH_0^K(\setS^0)\ar[r]\ar[d]_\cong&\calH_0^{\quot KH}(\setS^0)\ar[d]^\cong\\
          A(K)\ar[r]&A(\quot KH),
          }
\]
where the horizontal maps are restrictions and the vertical maps are Pontryagin-Thom maps, Wirthm\"{u}ller isomorphisms and suspensions, respectively. The lowest horizontal map 
sends a generator $[\quot KL]$ to $[(\quot KL)^H]$. This map is a ring homomorphism, so it sends units to units. Thus, if we start with the fundamental class in the upper left, we 
end up with a unit in the lower left by definition, and this maps to a unit in the lower right. Therefore, the restriction of the fundamental class maps to a unit via the 
vertical map on the right, which characterizes it as a fundamental class.

iii) $\implies$ iv): This follows from the equivalence of i) and ii).

iv) $\implies$ i): The fundamental class is determined by the property that it maps to a unit under
\[
 \calH_V^G(M_+)\to\calH_V^G(G_+\wedge_H\setS^N)\cong\calH_V^H(\setS^V)\cong A(H)\cong\colim_W\{\setS^W,\setS^W\}_H.
\]
An element in $\colim_W\{\setS^W,\setS^W\}_H$ is a unit if and only if its restriction to the colimit of the sets $\{\setS^{W^L},\setS^{W^L}\}$ is a unit for every subgroup $L$ of 
$H$. From this, the claim follows immediately.
\end{proof}

We can proceed to show that the equivariant Lefschetz number behaves well under restriction.

\begin{corollary}\label{prop:lefschetzrestriction}
Let $M$ be a $G$-manifold with finite orbit type and let $H\subseteq G$ be a subgroup. Let $\eta_H:A(G)\to A(W(H))$ be the fixed point homomorphism. Then
 \[
  \eta_H(L_G(f))=L_{W(H)}(f^H).
 \]
\end{corollary}

\begin{proof}
This follows in the case where $M$ is a compact orientable $V$-manifold from the commutative diagram
\[
 \xymatrix{
           \calH_V^G(M_+)\ar[rr]^{\Delta^*\tensor(\id, f)^*\hspace*{2.75cm}}\ar[d]&&\calH_V^G((M\times M)_+)\tensor\calH_V^G((M\times M)_+)\ar[d]\\
           \calH_V^G(M^H_+)\ar[rr]^{\Delta_H^*\tensor(\id, f^H)^*\hspace*{2.75cm}}&&\calH_{V^H}^{W(H)}((M^H\times M^H)_+)\tensor\calH_{V^H}^{W(H)}((M^H\times M^H)_+),
          }
\]
where the vertical maps are restriction maps. The fundamental class $\calO_M$ of $M$ maps to $\calO_\Delta\tensor\calO_{\Gamma_f}$ under the upper horizontal map. Since $\calO_M$
restricts to $\calO_{M^H}$ by Proposition \ref{prop:fundamentalrestriction}, the restriction of $\calO_\Delta\tensor\calO_{\Gamma_f}$ is $\calO_{\Delta_H}\tensor\calO_{\Gamma_{f^H}}$, 
where $\Delta_H$ denotes the diagonal of $M^H$. From this, the formula for the Lefschetz numbers follows immediately.   

The case for more general $M$ follows similarly by writing down the defining diagrams of the Lefschetz number, applying restriction and using the fact that Thom classes restrict 
to Thom classes, compare the remarks before Proposition \ref{prop:fundamentalrestriction}.
\end{proof}

We proceed to consider some special cases of Thom classes. Namely, assume that $P\subseteq M$ is an invariant $G$-orientable submanifold of the $G$-orientable $V$-manifold $M$
and moreover that $P$ is a $W$-manifold for some $W\subseteq V$, a $G$-subrepresentation. Let $U$ be an invariant tubular neighbourhood of $P$. We can regard this as a 
$V-W$-bundle $\nu^M_P:U\to P$, and this bundle has a Thom class $\tau$ living in $H^{V-W}_G(T\nu^M_P)$. Now we have the Pontryagin-Thom map $M_+\to T\nu^M_P$, and we denote 
the image of $\tau$ under the induced map
\[
 H^{V-W}_G(T\nu^M_P)\to H^{V-W}_G(M_+)
\]
by $\tau^M_P$ and call it the Thom class of $P$ in $M$.

\begin{lemma}\label{lem:thomclassnormalbundle}
 Let $M$ be an orientable $V$-manifold and $P, Q$ be orientable $Z_P$- and $Z_Q$-submanifolds, respectively. Assume that $P\cap Q$ is a $Z$-manifold with $V-Z_P=Z_Q-Z\in RO(G)$ and such 
 that 
 \[
  {i^P_{P\cap Q}}^*(\nu_P^M)=\nu^Q_{P\cap Q}
 \] 
 is a normal bundle for the inclusion $i^Q_{P\cap Q}:P\cap Q\to Q$. Then if $\tau^M_P$ is a Thom class for $P$ in $M$, ${i^M_Q}^*(\tau^M_P)$ is a Thom class for $P\cap Q$ in $Q$. 
\end{lemma}

\begin{proof}
Firstly, let $\varphi:\quot GH\to P\cap Q$ be any $G$-map. $i^P_{P\cap Q}\circ\varphi:\quot GH\to P$ is a $G$-map making the diagram
\[
 \xymatrix{
          H^{V-Z_P}_G(T\nu^M_P)\ar[r]^{T\varphi^*T{i^P_{P\cap Q}}^*}\ar[d]_{T({i^P_{P\cap Q}})^*}&H^{V-Z_P}_G(T(\varphi^*{i^P_{P\cap Q}}^*(\nu^M_P)))\ar[d]^\cong\\
          H^{V-Z_P}_G(T({i^P_{P\cap Q}}^*(\nu^M_P)))\cong H^{Z_Q-Z}_G(T\nu^Q_{P\cap Q})\ar[r]^{\hspace*{1.5cm}T\varphi^*}&H^{Z_Q-Z}_G(T(\varphi^*(\nu^Q_{P\cap Q})) 
          }
\]
commutative. This shows that the image of the Thom class of $\nu^M_P$ is a Thom class for $\nu^Q_{P\cap Q}$. 

Secondly, let $U\subseteq M$ be an invariant tubular neighbourhood for $P$ in $M$ such that $\nu^M_P:U\to P$ is the bundle projection of the normal bundle. By assumption, the pullback of 
this bundle via $i^P_{P\cap Q}$ is a normal bundle for $P\cap Q$ in $Q$, which implies that the diagram
\[
 \xymatrix{
          T(\nu^Q_{P\cap Q})\ar[d]_{T({i^P_{P\cap Q}})}&Q_+\ar[l]\ar[d]^{i^M_Q}\\
          T(\nu^M_P)&M_+\ar[l]
          }
\]
commutes. It follows that in cohomology,
\[
 \xymatrix{
          H^{V-Z_P}_G(T(\nu^Q_{P\cap Q}))\ar[r]&H^{V-Z_P}_G(Q_+)\\
          H^{V-Z_P}_G(T(\nu^M_P))\ar[u]^{T({i^P_{P\cap Q}})^*}\ar[r]&H^{V-Z_P}_G(M_+)\ar[u]_{{i^M_Q}^*}
          }
\]
commutes. Starting with the Thom class in the lower left, we have $\tau^M_P$ in the lower right, which maps up to ${i^M_Q}^*(\tau^M_P)$. In the upper left we have the Thom class for
$\nu^Q_{P\cap Q}$, which by definition maps to $\tau^Q_{P\cap Q}$ under the upper horizontal map. Thus, the identity ${i^M_Q}^*(\tau^M_P)=\tau^Q_{P\cap Q}$ is proven.
\end{proof}

In Proposition \ref{prop:thomclassproduct}, we have calculated that cross products of Thom classes are Thom classes. A similar result holds for the cup product and Thom classes 
of submanifolds. We need the following variation of the Thom diagonal. Let $P\subseteq M\subseteq W$ be equivariant embeddings of invariant submanifolds. We find an invariant 
tubular neighbourhood $U^W_P$ of $P$ in $W$, an invariant tubular neighbourhood $U^M_P\subseteq U^W_P$ of $P$ in $M$ and an invariant tubular neighbourhood $U^W_M$ of $M$ in $W$. 
Let $s:U^W_M\to M$ be the tubular retraction. Then we can choose the neighbourhoods in a way that, if $z\in\partial U^W_P$, then either $z\notin U^W_M$, or $s(z)\notin U^M_P$.
For example, we can choose $U^W_P$ such that it contains $s^{-1}(\conj{U^M_P})$. With this choice, the map
\[
 d:T\nu^W_P\to T\nu^M_P\wedge T\nu^W_M,\;x\mapsto[s(x), x],
\]
is well defined.

\begin{lemma}\label{lem:cupproductthomclass}
 Let $M$ be an orientable $V$-manifold, $P\subseteq M$ an orientable $Z$-submanifold and $M\to W$ an embedding of $M$ into a $G$-representation $W$. Let $\nu^M_P$ be a normal bundle for the 
 embedding $P\to M$, $\nu^W_M$ a normal bundle for the embedding $M\to W$ and let $\nu^W_P$ be the induced normal bundle of the embedding $P\to W$. Let
 \[
  \tau^M_P\in H^{V-Z}_G(T\nu^M_P),\;\tau^W_M\in H^{W-V}_G(T\nu^W_M)
 \]
 be Thom classes for the normal bundles $\nu^M_P$, $\nu^W_M$, respectively. Then $\tau^M_P\cup\tau^W_M\in H^{W-Z}_G(T\nu^W_P)$ is a Thom class for the bundle $\nu^W_P$, where the cup product
 is taken with respect to the map
\[
 d:T\nu^W_P\to T\nu^M_P\wedge T\nu^W_M,\;x\mapsto[s(x), x]
\]
 defined above.
\end{lemma}

\begin{proof}
Let $\varphi:\quot GH\to P$ be a $G$-embedding. $\varphi$ induces maps $T\varphi^W_P:T(\varphi^*(\nu^W_P))\to T\nu^W_P$, $T\varphi^M_P:T(\varphi^*(\nu^M_P))\to T\nu^M_P$ and 
$T\varphi^W_M:T(\varphi^*(\nu^W_M))\to T\nu^W_M$. These fit together, after identification of the Thom spaces over orbits, in the diagram
\[
 \xymatrix{
           T\nu^W_P\ar[r]^{d\quad\quad}&T\nu^M_P\wedge T\nu^W_M\\
           G_+\wedge_H\setS^{W-Z}\ar[r]^{d'\quad\quad}\ar[u]^{T\varphi^W_P}&G_+\wedge_H\setS^{V-Z}\wedge G_+\wedge_H\setS^{W-V}\ar[u]_{T\varphi^M_P\smash T\varphi^W_M},
          }
\]
where $d'$ is induced by the diagonal map on $G_+$ and an identification of the spheres.

This implies commutativity of
\[
 \xymatrix{
  H^{V-Z}_G(T\nu^M_P)\tensor H^{W-V}_G(T\nu^W_M)\ar[r]^{\hspace*{1.75cm}\cup_d}\ar[d]&H^{W-Z}_G(T\nu^W_P)\ar[d]\\
  H^{V-Z}_G(G_+\wedge_H\setS^{V-Z})\tensor H^{W-V}_G(G_+\wedge_H\setS^{W-V})\ar[r]^{\hspace*{1.75cm}\cup_{d'}}&H^{W-Z}_G(G_+\wedge_H\setS^{W-Z}).
  }
\]
Again, we can complete the diagram using Wirthm\"{u}ller and suspension isomorphisms together with the fact that the cup product respects both of these to obtain a commutative 
diagram
\begin{small}
\[
\xymatrix{
          H^{V-Z}_G(T\nu^M_P)\tensor H^{W-V}_G(T\nu^W_M)\ar[r]^{\hspace*{1.75cm}\cup_d}\ar[d]&H^{W-Z}_G(T\nu^W_P)\ar[d]\\
          H^{V-Z}_G(G_+\wedge_H\setS^{V-Z})\tensor H^{W-V}_G(G_+\wedge_H\setS^{W-V})\ar[r]^{\hspace*{1.75cm}\cup_{d'}}\ar[d]&H^{W-Z}_G(G_+\wedge_H\setS^{W-Z})\ar[d]\\
          H^{V-Z}_H(\setS^{V-Z})\tensor H^{W-V}_H(\setS^{W-V})\ar[r]^\times\ar[d]& H^{W-Z}_H(\setS^{W-Z})\ar[d]\\
          A(H)\tensor A(H)\ar[r]^\mu\ar[r]&A(H).
          }
\]
\end{small}
Starting with the product $\tau^M_P\tensor\tau^W_M$ in the upper left, this maps to a product of units in $A(H)\tensor A(H)$ and thus to a unit in the lower right, implying that
$\tau^M_P\cup_d\tau^W_M=\tau^W_P$ is a Thom class.
\end{proof}

It is true non-equivariantly that Thom classes of embeddings of submanifolds are dual to the fundamental classes of the submanifold. We can use our knowledge of Thom classes 
of submanifolds to establish this result in the equivariant world, thereby justifying the definition of the equivariant Lefschetz number for general $G$-manifolds. 

\begin{lemma}\label{lem:fundamentalthomduality}
 Let $P\subseteq M$ be a $Z$-submanifold of the $V$-manifold $M$. Let $\tau^M_P$ be a Thom class of $P$ in $M$. Let $\calO^M_P$ be the image of its corresponding fundamental 
 class $\calO_P$ under the inclusion $P\to M$. Then
\[
 \tau^M_P\cap\calO_M=\calO^M_P.
\]
\end{lemma}

\begin{proof}
Again we have to fix the various tubular neighbourhoods first. $P$ embeds into $W$ via the composition of the embeddings of $P$ in $M$ and of $M$ in $W$. We find invariant tubular 
neighbourhoods $U$ of $M$ in $W$ and $U'$ of $P$ in $M$, together with retractions $s:U\to M$, $r:U'\to P$. We can arrange this such that $s^{-1}(U')=U''$ is a neighbourhood of 
$P$ in $W$ contained in $U$ and $r\circ s:U''\to P$ is the tubular retraction. We obtain a Pontryagin-Thom map for the inclusion $U''\to U$, i.e. a map 
$\psi:T\nu^W_M\to T\nu^W_P$.

We work on the represented level. Using the fact that the map dual to the inclusion is the Pontryagin-Thom map, which follows from III.3.7 in \cite{may}, the usual
duality diagram as in Proposition \ref{prop:fundamentalthomclass} becomes
\[
 \xymatrix{
          \calH_Z^G(P_+)            \ar[r]^{i_*}                    \ar[d]^\cong         &  \calH_Z^G(M_+)\ar[d]^\cong            \\
          H^{W-Z}_G(T\nu^W_P)\ar[r]^{\psi^*}                                               &  H^{W-Z}_G(T\nu^W_M),                                        
          }
\]
We now start with the bottom row and complete it to the following diagram.
\[
 \xymatrix{
          H^{W-Z}_G(T\nu^W_P)\ar[r]&H^{W-Z}_G(T\nu^W_M)\\
          H^0_G(P_+)\ar[u]^{\cup t^W_P}\ar[ru]^{\cup t^W_P}\ar[d]_{\cup t^M_P}\ar[dr]^{\cup t^M_P}&\\
          H^{V-Z}_G(T\nu^M_P)\ar[r]&H^{V-Z}_G(M_+)\ar[uu]^{\cup t^W_M}
          }
\]
The cup products involved come from different diagonals. The upper left triangle commutes due to commutativity of the diagram
\[
 \xymatrix{
          T\nu^W_M\ar[r]\ar[d]^\psi&P_+\wedge T\nu^W_P\ar[d]^{\id\wedge \id}\\
          T\nu^W_P\ar[r]&P_+\wedge T\nu^W_P.
          }
\]
The upper horizontal map is the modified Thom diagonal, with the $\psi$ in the second component and the retraction onto $P$ in the first. The lower horizontal map is the actual Thom 
diagonal. Naturality of the cup product implies the commutativity of the first triangle.

The lower left triangle commutes by the same reasoning, using the diagram
\[
 \xymatrix{
          M_+\ar[r]\ar[d]&P_+\wedge T\nu^M_P\ar[d]^{\id\wedge\id}\\
          T\nu^M_P\ar[r]&P_+\wedge T\nu^M_P.
          }
\]
Thus, we see that an element $x\in H^0_G(P_+)$ maps to $x\cup t^W_P$ in the upper right. The various diagonals involved fit together in an associativity diagram as in Proposition 
\ref{prop:cupassociativity}. That Proposition therefore implies that mapping $x$ to $x\cup t^M_P$ in the lower right and then with the Thom isomorphism to
$(x\cup t^M_P)\cup t^W_M$ gives $x\cup(t^M_P\cup t^W_M)$. By Lemma \ref{lem:cupproductthomclass}, this is equal to $x\cup t^W_P$, so the whole diagram commutes. The two maps on the left 
are both Thom isomorphisms, so they compose to give an isomorphism which maps the Thom classes to one another.

Pasting the two diagrams together yields
\[
 \xymatrix{
          \calH_Z^G(P_+) \ar[r]^{i_*} \ar[d]^\cong         &  \calH_Z^G(M_+)\ar[d]^\cong            \\
          H^{W-Z}_G(T\nu^W_P)\ar[r]^{\psi^*}\ar[d]^\cong   &  H^{W-Z}_G(T\nu^W_M) \ar[d]^\cong\\
          H^{V-Z}_G(T\nu^M_P)\ar[r]                        &  H^{V-Z}_G(M_+),
          }
\]
and the lower horizontal map is induced by the Pontryagin-Thom map of $P\subseteq M$. The composition up the right is the cap product with the fundamental class, which follows 
from Proposition \ref{prop:fundamentalthomclass} and the definition of the cap product. It therefore follows that starting with the fundamental class of $P$ in the upper left, 
we end up with the Thom class of the embedding $P\to M$ in the lower left, and by definition, we end up with $\tau^M_P$ in the lower right. Then going up gives 
$\tau^M_P\cap\calO_M$, and this is equal, by commutativity of the diagram, to $i_*(\calO_P)=\calO^M_P$.
\end{proof}

\section{Properties of the equivariant Lefschetz number}

We start to give the desired geometric proof of the equivariant Lefschetz fixed point theorem. 

\begin{theorem}\label{thm:lefschetz}
 Let $M$ be a compact orientable $V$-manifold. Then if a $G$-map $f:M\to M$ has no fixed point of orbit type at least $(H)$, we have $\eta_H(L_G(f))=0$.
\end{theorem}

\begin{proof}
 It obviously suffices to prove the theorem for $H=e$, in the general case, we just replace $G$ by $W(H)$ and use Proposition \ref{prop:lefschetzrestriction}. So we have to show that if 
 $f$ has no fixed points, $L_G(f)=0$. If $f$ has no fixed points, the intersection of the diagonal $\Delta\subseteq M\times M$ with the graph $\Gamma$ of $f$ is empty.

 Both $\Delta$ and $\Gamma$ are orientable $V$-manifolds, fundamental classes are given by the images of the fundamental class of $M$ via the canonical homeomorphisms between $M$ and each
 of these manifolds.

 We denote by $\calO_\Delta$ the image of the fundamental class of $\Delta$ under the diagonal embedding $\Delta\to M\times M$, similarly $\calO_\Gamma$ the image of the fundamental
 class of $\Gamma$ under the embedding $\Gamma\to M\times M$.

 We claim that $\calO_\Delta\bullet\calO_\Gamma=0$. Indeed, let $\tau_\Delta$, $\tau_\Gamma$ be the classes dual to $\calO_\Delta$ and $\calO_\Gamma$, i.e.
 \[
  \tau_\Delta\cap\calO_{M\times M}=\calO_\Delta,\;\tau_\Gamma\cap\calO_{M\times M}=\calO_\Gamma.
 \]
 Then by Lemma \ref{lem:fundamentalthomduality}, $\tau_\Delta$ and $\tau_\Gamma$ are just images of the Thom classes of normal bundles to the respective embeddings of the 
 manifolds $\Delta$ and $\Gamma$ under the Pontryagin-Thom maps. Let $\delta:M\to M\times M$ be the diagonal embedding. We can calculate
 \begin{eqnarray*}
  \calO_\Delta\bullet\calO_\Gamma&=&P(P^{-1}(\calO_\Gamma)\cup P^{-1}(\calO_\Gamma))\\
                                 &=&(\tau_\Gamma\cup\tau_\Delta)\cap\calO_{M\times M}\\
                                 &=&\tau_\Gamma\cap(\tau_\Delta\cap\calO_{M\times M})\\
                                 &=&\tau_\Gamma\cap(\delta_*(\calO_M))\\
                                 &=&\delta_*(\delta^*(\tau_\Gamma)\cap\calO_M)\\
                                 &=&0,
 \end{eqnarray*}
 since $\delta^*(\tau_\Gamma)=0$ by Lemma \ref{lem:thomclassnormalbundle}.
\end{proof}

Next, let us compare the equivariant Lefschetz number we defined with the existing notion of \cite{rosenberg}. The fundamental class $\calO$ of a $V$-manifold $M$ is uniquely 
characterized by the property that its image under the restriction $\calH_V^G(M)\to H_{\abs{V^H}}(M^H)$ is a fundamental class for $M^H$ for every subgroup of $G$. 
Following this restriction map through the defining diagram of the Lefschetz number, we obtain
\[
 \xymatrix{
 \calH_V^G(M)\tensor\calH_V^G(M)\ar[rr]\ar[d]&&\calH_0^G(*)\ar[r]^\cong\ar[d]& A(G)\ar[d]\\
 H_{\abs{V^H}}(M^H)\tensor H_{\abs{V^H}}(M^H)\ar[rr]&&H_0(*)\ar[r]^\cong&\setZ,
          }
\]
and the map on the right is given by sending a generator $[\quot GK]$ to $\abs{(\quot GK)^H}$ or equivalently, sending a stable $G$-map $f:\setS^W\to\setS^W$ to the degree of 
$f^H:\setS^{\abs{W^H}}\to\setS^{\abs{W^H}}$. Since the lower row determines the non-equivariant Lefschetz number $L(f^H)$ of $f^H$, the equivariant Lefschetz number is an element 
of $A(G)$ with the property that its non-equivariant restriction to $H$-fixed points is $L(f^H)$. 

Similarly one can check that this holds true for general $G$-manifolds $M$ and self maps $f:M\to M$ such that the equivariant Lefschetz number is defined. It is shown in 
\cite{tomdieck} that this property uniquely determines the element of $A(G)$. In case that $G$ is finite it therefore follows immediately from definition (4.1) of
the equivariant Lefschetz class in \cite{rosenberg}, that the two equivariant Lefschetz numbers are the same.

We now want to give a more explicit formula for the equivariant Lefschetz number. In order to do this, we make several generic assumptions and simplifications.

\begin{enumerate}[i)]
 \item Firstly, if $f:M\to M$ is a $G$-map, we can embed $M$ into a $G$-representation $V$. In particular, every $G$-orbit in $M$ embeds into $V$. Then the Lefschetz number of $f$
 equals the Lefschetz number of the map $U\to U,\;x\mapsto i\circ f\circ r$, where $U$ is a tubular neighbourhood of $M$ in $V$, $r:U\to M$ the tubular retraction and $i:M\to U$
 the embedding. Thus, we can assume that $M$ is a $V$-manifold such that every $G$-orbit in $M$ embeds into $V$.

 \item The equivariant Lefschetz number is determined via the diagram
\[
 H^V_G(T\nu^{M\times M}_\Delta)\slra{(\id, f)^*}H^V_G(\quot{\conj N}{\partial N})\slra{\cong}\calH_0^G(N_+)\to\calH_0^G(\setS^0),
\]
so it is clear that it is determined by local data. Realizing the Thom space of the diagonal embedding as an arbitrary small neighbourhood of the diagonal, $N$ will be an 
arbitrarily small neighbourhood of the fixed point set of $f$. 

So in order to understand the equivariant Lefschetz number, we need to understand Lefschetz numbers of maps $f:G\times_HW\to G\times_HW$, where $G\times_HW$ is the tubular 
neighbourhood in $V$ of an embedding of the orbit $\quot GH$ and $G\times_H0$ is the unique fixed orbit of $f$.

\item We can furthermore assume that $f$ is smooth and the fixed orbit is non-degenerate, see e.g. \cite{field} or \cite{wruck}. This means that the derivative of $f$ at $[e, 0]$
in direction normal to the orbit has no eigenvalue of unit modulus. This assumption implies that there are no points in a neighbourhood of $G\times_H0$ with $f([g, w])=[h, w]$ 
except for the ones with $w=0$.

\item By Krupa's normal decomposition lemma, see e.g. \cite{field1}, Lemma 6.2, $f$ is of the form
\[
 G\times_HW\to G\times_HW,\;[g,w]\mapsto (g\gamma(w),n(w)),
\]
where $\gamma:W\to G$ is $H$-equivariant with respect to the conjugation action on $G$ and $n:W\to W$ is an $H$-map. 

\item Let $A\subseteq\quot GH$ be an $H$-invariant neighbourhood of $[e]$, $H$-isomorphic to 
$L=T_{[e]}{\quot GH}$. By Lemma 3.10.2 of \cite{field}, we can assume that $A$ is $N(H)$-invariant and there exists an $N(H)$-equivariant section $\sigma:A\to G$, i.e. 
$\sigma([hg])=h\sigma([g])h^{-1}$ for $h\in N(H)$. 

The set
\[
 U=\{([g, w], [g\sigma([g']), w+w'])\;|\;g\in G, [g']\in A, w\in W, w'\in\setB_1(W)\}
\]
is the total space of a normal bundle of the diagonal embedding of $G\times_HW$ for any ball $\setB_1(W)$ around $0$ in $W$. We can choose $\setB_1(W)$ such that we find another 
ball $\setB_2(W)$ around $0$ in $W$ with $([g, w], f([g, w]))\notin U$ for $w\in\setS_2(W)=\partial\setB_2(W)$. This is possible since $G\times_H0$ is the only fixed orbit of $f$. 
By our non-degeneracy assumption, we can arrange the neighbourhoods in a way such that
\[
 ([g, w], [g\eta(tw), n(w)])\notin U
\]
for $w\in\setS_2(W)$ and $t\in\incc{0,1}$. This shows that the map
\[
 (\id, f):\quot{\setB_2(W)}{\setS_2(W)}\to\quot U{\partial U},\;w\mapsto\begin{cases}
                                                                         * & w\in\setS_2(W)\\
                                                                         * & ([g, w], f([g, w]))\notin U\\
                                                                         ([g, w], f([g, w])) & \mbox{ else}
                                                                        \end{cases}
\]
is equivariantly homotopic to the map
\[
 \quot{\setB_2(W)}{\setS_2(W)}\to\quot U{\partial U},\;w\mapsto\begin{cases}
                                                                         * & w\in\setS_2(W)\\
                                                                         * & ([g, w], [g, n(w)])\notin U\\
                                                                         ([g, w], [g, n(w)]) & \mbox{ else}
                                                                        \end{cases}.
\]
So we can assume that $(\id, f)$ already has this property. We note in particular that the homotopy only involves the group coordinate, so the derivative of $n$ at $0$ equals the 
derivative of $f$ at $[e, 0]$ in normal direction to the group orbit. 
\end{enumerate}

Summarizing the assumptions and simplifications, we are left with the task to compute the Lefschetz number of a map $G\times_HW\to G\times_HW,\;[g,w]\mapsto[g, n(w)]$ of
a tubular neighbourhood of an embedding of $\quot GH$ into a $G$-representation $V$, where $T_0n$ has no eigenvalue of unit modulus. We will continue to use the local data 
assembled in v) above.

The equivariant Lefschetz number is defined as the image of the Thom class under
\[
 H^V_G(\quot U{\partial U})\slra{(\id, f)^*}H^V_G(G_+\smash_H\setS_2^W)\slra{\cong}\calH_0^G((G\times_H\setB_2(W))_+)\to\calH_0^G(\setS^0),
\]
the middle isomorphism being Poincar\'{e} duality. 

We note that the Thom space $\quot U{\partial U}$ can be identified with $G_+\smash_H\setB_1(W)_+\smash\setS^L\smash\setS^W$ via
\[
 ([g, w], [g\sigma([g']), w+w'])\mapsto[g, w, \zeta([g']), w'],
\]
where $\zeta:A\to L$ is an $H$-homeomorphism satisfying $\zeta([e])=0$.

The special form of $(\id, f)$ then has the property that the diagram
\[
 \xymatrix{
  G_+\smash_H\setS^W\ar[r]^{(\id, f)\quad\quad\quad}\ar[dr]_{i} & G_+\smash_H\setB_1(W)_+\smash\setS^L\smash\setS^W\\
                                                          & G_+\smash_H\setS^L\smash\setS^W\ar[u]_{\Sigma^L(\id, \tilde{n})}
          }
\]
commutes, where $i$ is the inclusion and $(\id, \tilde{n})$ denotes the map
\[
 G_+\smash_H\setS^W\to G_+\smash_H\setB_1(W)_+\smash\setS^W,\;[g,w]\mapsto ([g, w], n(w)-w).
\]
Thus, we obtain a diagram
\[
 \xymatrix{
          H^V_G(G_+\smash_H\setB_1(W)_+\smash\setS^L\smash\setS^W)\ar[r]^{\quad\quad(\id, f)^*}\ar[dr]_{\Sigma^L(\id, \tilde{n})^*}\ar[dd]^\cong&H^V_G(G_+\smash_H\setS^W)\\
          &H^V_G(G_+\smash_H\setS^L\smash\setS^W)\ar[d]^\cong\ar[u]^{i^*}\\
          H^V_H(\setB_1(W)_+\smash\setS^L\smash\setS^W)\ar[r]^{\Sigma^L(\id, n)^*}\ar[d]^\cong&H^V_H(\setS^L\smash\setS^W)\ar[d]^\cong\\
          H^W_H(\setB_1(W)_+\smash\setS^W)\ar[r]^{(\id, n)^*}&H^W_H(\setS^W).
          }
\]
We will see later, when we compute the $H$-equivariant Lefschetz number of $n$, that $\setB_1(W)_+\smash\setS^W$ is a model for the Thom space of the diagonal embedding of 
$\setB_1(W)$ and the corresponding map $(\id, n)$ indeed makes the diagram commutative. 
 
We complete this diagram by inserting Poincar\'{e} duality isomorphisms on the right. We obtain, with some suppressed suspension isomorphisms,
\[
 \xymatrix{
  H^V_G(G_+\smash_H\setS^W)\ar[r]^P&\calH_V^G({\quot GH}_+\smash\setS^V)\ar[r]^{p_*}&\calH_V^G(\setS^V)\\
  H^V_G(G_+\smash_H\setS^W\smash\setS^L)\ar[r]^P\ar[u]^{i^*}&\calH^G_V(G_+\smash_H\setS^W)\ar[u]^{j_*}&\\
  H^V_H(\setS^W\smash\setS^L)\ar[r]^P\ar[u]^w&\calH_V^H(\setS^V)\ar[u]^w\ar[ruu]_?&\\
  H^W_H(\setS^W)\ar[r]^P\ar[u]^\cong&\calH_W^H(\setS^W)\ar[u]^\cong&,        
  }
\]
where $j:G_+\smash_H\setS^W\to{\quot GH}_+\smash\setS^V$ is the map $[g, w]\mapsto([g], g.w)$. The two diagrams together imply that the $G$-equivariant Lefschetz number of $f$ 
equals the image of the $H$-equivariant Lefschetz number of $n$ under the map $p_*\circ i_*\circ w$, where $w$ is the Wirthm\"{u}ller isomorphism in the middle right above. 
Now it follows from the definition of the transfer $t:\setS^V\to{\quot GH}_+\smash\setS^V$, \cite{may} Definition II.6.15, and the construction of the Wirthm\"{u}ller isomorphism, 
\cite{may} Definition II.6.1, that the diagram
\[
 \xymatrix{
 \{\setS^V,\setS^V\smash\calH\}_H\ar[d]^{\cong}\ar[r]^w&\{\setS^V,G_+\smash_H\setS^W\smash\calH\}_G\ar[d]^{(j\smash\id)_*}\\
 \{{\quot GH}_+\smash\setS^V,\setS^V\smash\calH\}_G\ar[d]^{t^*}&\{\setS^V, {\quot GH}_+\smash\setS^V\smash\calH\}_G\ar[dl]^{(p\smash\id)_*}\\
 \{\setS^V,\setS^V\smash\calH\}_G&
}
\]
commutes. The vertical map on the left induces the induction map $t^G_H:A(H)\to A(G)$. Consequently, the map labelled with a question mark in the diagram above 
induces the induction map $t^G_H$. We conclude that for the $G$-equivariant Lefschetz number of $f$, the formula
\[
 L_G(f)=t^G_H(L_H(n))
\]
holds.

It therefore remains to calculate the $H$-equivariant Lefschetz number of the map $n$. This is an $H$-map of an $H$-representation with $0$ as a unique fixed point. We emphasize that
under our assumptions, $\id-T_0n$ is invertible.

We find two balls $\setB_1=\setB_1(W),\setB_2=\setB_2(W)$ in $W$ around $0$ such that $\frac{w-n(w)}2\notin\setB_2$ for $w\in\setS_1=\partial\setB_1$. We take
\[
 U=\{(w, w)+(z, -z)\;|\;w\in W, z\in\setB_2\}
\]
as a tubular neighbourhood of $\Delta(W)$ in $W\times W$ with the obvious tubular retraction. The Thom space $T\nu^{W\times W}_\Delta$ is the space $\quot{\conj U}{\partial U}$. 
The map $(\id, n)$ defined in the definition of the equivariant Lefschetz number is induced by
\[
 N:w\mapsto(w, n(w))=\left(\frac{w+n(w)}2,\frac{w+n(w)}2\right)+\left(\frac{w-n(w)}2,\frac{n(w)-w}2\right)\in U.
\]
In detail,
\[
 (\id, n):\quot{\conj{\setB}_1}{\setS}_1\to\quot{\conj U}{\partial U},\;(\id,n)(w)=\begin{cases}
           * & \frac{w-n(w)}2\notin\setB_2\\
           N(w) & \mbox{ else}.
          \end{cases}
\]
The projection 
\[
p_2:\quot{\conj U}{\partial U}\to\quot{\conj{\setB}_2}{\setS_2},\;(w,w)+(z, -z)\mapsto z
\]
is an $H$-homotopy equivalence, since $W$ is $H$-contractible. Therefore, the defining diagram for the equivariant Lefschetz number extends to
\[
 \xymatrix{
          H^W_H(\quot{\conj U}{\partial U})\ar[r]^{(\id,n)^*}&H^W_H(\quot{\conj{\setB_1}}{\setS_1})\ar[r]^P\ar[d]^\id&\calH_0^H({\setB_1}_+)\ar[r]^\eps&\calH_0^H(\setS^0)\\
          H^W_H(\quot{\conj{\setB_2}}{\setS_2})\ar[u]^{p_2^*}\ar[r]\ar[d]^\cong&H^W_H(\quot{\conj{\setB_1}}{\setS_1(W)})\ar[d]^\cong&&\\
          H^W_H(\setS^W)\ar[r]&H^W_H(\setS^W)&&
          }
\]
The horizontal map in the middle is induced by
\[
 N':\quot{\conj{\setB}_1}{\setS_1}\to\quot{\conj{\setB}_2}{\setS_2},\;w\mapsto\begin{cases}
          * & \frac{w-n(w)}2\notin\setB_2\\
          \frac{w-n(w)}2 & \mbox{ else}. 
         \end{cases}
\]
This is, up to the identification $\quot{\conj{\setB}_1}{\setS_1}\cong\setS^W$, the Pontryagin-Thom construction $\setS^W\to\setS^W$ of the map $\id-n:W\to W$ (see e.g. 
\cite{bredon}). Using the representing spectra, it is clear that the diagram
\[
\xymatrix{
H^W_H(\setS^W)\ar[d]^{\sigma^{-1}}\ar[r]^P&\calH_0^H({\setB_1}_+)\ar[r]^{\eps_*}&\calH_0^H(\setS^0)\ar[d]^\cong\\
H^0_H(\setS^0)\ar[rr]^\cong&&A(H)
}
\]
commutes. In the initial diagram for the equivariant Lefschetz number, the Thom class in $H^W_H(\quot{\conj U}{\partial U})$ maps to the unit element in $H^W_H(\setS^W)$ under 
the vertical isomorphism on the left. Then going right is just precomposition with $N'$, so the unit is mapped to the class of $N'$ in $\{\setS^0,\setS^0\}_H\cong H^W_H(\setS^W)$. 
It follows that the equivariant Lefschetz number is the stable equivariant homotopy class of the Pontryagin-Thom construction for the map $\id-n$ (after identifying $A(H)$ with 
$\{\setS^0,\setS^0\}_H$).

Under the assumption that $\id-T_0n$ is invertible, the Pontryagin-Thom constructions of $\id-T_0n$ and $\id-n$ are $H$-homotopic. Indeed, we just have to show
that there exists a ball $\setB(W)$ in $W$ and a sphere $\setS(W)$, both centered at $0$, such that $w-h(t, w)\notin\setB(W)$ for $w\in\setS(W)$, where
\[
 h(t, w)=w-(t\cdot n(w)+(1-t)\cdot T_0n(w)).
\]
Then $h$ induces an $H$-homotopy of the Pontryagin-Thom constructions. The derivative of $h_t=h(t,\,\cdot\,)$ is given by $\id-T_0n$, so $h_t$ is invertible in a small ball 
around $0$. Since $h_t(0)=0$, this shows that $\norm{h_t(w)}>\eps>0$ for all $t$, all $w$ in a small sphere around $0$ and some $\eps>0$. The claim follows.

We denote the element of $A(G)$ corresponding to the Pontryagin-Thom construction of a map $f:V\to V$ by $Deg_G(f)$. We have proven the following theorem.

\begin{theorem}\label{thm:lefschetz2}
 Let $M$ be a $G$-manifold and $f:M\to M$ a $G$-map with finitely many $G$-orbits of fixed points $Gx_1,\dots, Gx_n$. Assume that $\id-N_{x_i}f$ has no eigenvalue of unit modulus
 for $i=1, \dots, n$, where $N_xf$ is the component of $T_xf$ normal to the orbit $Gx$. Then the equivariant Lefschetz number $L_G(f)$ is given as
 \[
  L_G(f)=\sum_{i=1}^nt^G_{G_{x_i}}\left(Deg_{G_{x_i}}(\id-N_{x_i}f)\right).
\]
\end{theorem}


\section{The equivariant Fuller index}
In this section we construct an equivariant Fuller index, using similar homological techniques as for the equivariant Lefschetz number. For the basic theory of the Fuller index,
see \cite{chowmallet}, \cite{franzosa} and \cite{fuller}. Our approach is an equivariant generalization of \cite{franzosa}.

The initial definition on the path to an equivariant Fuller index resembles the definition of the equivariant Lefschetz number for non-compact manifolds. We recall that the 
equivariant Lefschetz number, in the case where $M$ is a compact orientable $V$-manifold, equals the image of the Thom class $\tau$ of the diagonal embedding of $M$ into 
$M\times M$ under the sequence of maps
\[
 H^V_G(T\nu^{M\times M}_\Delta)\slra{\psi^*}H^V_G((M\times M)_+)\slra{(\id,f)^*}H^V_G(M_+)\slra{\cap\calO_M}\calH_0^G(M)\slra{\eps}A(G),
\]
where the first map is induced by the Pontryagin-Thom map.

We now use ideas of Franzosa \cite{franzosa} to assign an index to an equivariant flow on a $V$-manifold, related to an isolated compact set of periodic points. A periodic point
of a flow $\varphi:M\times\setR\to M$ is a pair $(x, T)\in M\times\setR$ such that $\varphi(x, T)=x$.

We assume that $M$ is an orientable $V$-manifold and $C\subseteq M\times\setR$ is a compact subset of periodic points. Moreover, we assume that $C$ is isolated, i.e. there exists 
an open set $\Omega\subseteq M$ with compact closure, and $a, b\in\setR$, $a, b>0$ such that $C\subseteq\conj{\Omega}\times\inoo{a,b}$ and the only periodic points in 
$\Omega\times\incc{a,b}$ are the points in $C$. $\Omega$ can be chosen to be a manifold with boundary, and in particular, the image of $\partial\Omega$ under the map 
$(\pi_1, \varphi):M\times\setR\to M\times M$ does not meet the diagonal. We find a small invariant tubular neighbourhood $U$ of the diagonal such that $\partial\Omega$ does not 
meet $U$.

Then we can define a map $\Phi:\Sigma\quot\Omega{\partial\Omega}\to T\nu^{M\times M}_\Delta$, where $\nu^{M\times M}_\Delta$ is the normal bundle of the diagonal embedding with 
total space $U$. We rescale the suspension variable to run between $a$ and $b$, and define
\[
 \Phi:\Sigma\quot\Omega{\partial\Omega}\to T\nu^{M\times M}_\Delta,\;(x, t)\mapsto\begin{cases}
                                                       * & (x, t)\notin\Omega\times\inoo{a,b}\\
                                                       * & (x, \varphi(x, t))\notin U\\
                                                       (x, \varphi(x, t)) & \mbox{ else.} 
                                                      \end{cases}
\]
This is well defined by the choice of $U$. With this map in place, we imitate the definition of the equivariant Lefschetz class. We define the homological index of $\varphi$
with respect to $C$ to be the image of a Thom class of $T\nu^{M\times M}_\Delta$ under the sequence of maps
\[
 H^V_G(T\nu^{M\times M}_\Delta)\slra{\Phi^*}H^V_G(\Sigma\quot\Omega{\partial\Omega})\cong H^{V-1}_G(\quot\Omega{\partial\Omega})\slra{\cap\calO_\Omega}\calH_1^G(\Omega_+)
                               \to\calH_1^G(M_+),
\]
where the middle isomorphism is the suspension isomorphism and the last map is induced by the inclusion. 

There are obvious similarities with the equivariant Lefschetz number, but also significant differences. For one, the index lives in a group which depends on the manifold $M$, 
rather than a universal group like $A(G)$ in the case of the Lefschetz number. Furthermore, the index must be zero if $\calH_1^G(M_+)$ is trivial. Thinking non-equivariantly, 
this would mean that the index is zero for all simply connected manifolds. But certainly, one has interesting flows and periodic orbit structures on such manifolds as well.

The same problem arises for the ordinary Fuller index, and there is a standard construction by Fuller, generalized by Franzosa in \cite{franzosa}, which tackles this issue.
We first note that, as in the case of the Lefschetz number, we can drop orientability assumptions on $M$. If $\varphi$ is a flow on any $G$-manifold with finite orbit type and 
$C$ is an isolated compact set of periodic orbits, we can define an index for $\varphi$ with respect to $C$ as follows. We embed $M$ into a $G$-representation $V$ and let $U$ be 
an invariant tubular neighbourhood of the embedding. We can assume that $\varphi$ is induced by a vector field $\xi$ on $M$ and we can extend $\xi$ to a vector field $\xi_0$ on 
$U$ by defining $\xi_0(x)=\xi(r(x))-(x-r(x))$, where $r:U\to M$ is the tubular retraction and we identify tangential spaces with $V$ in the canonical way. Then $\xi_0$ is a 
vector field on $U$ with the same periodic orbits as $\xi$. In particular, $C$ is still an isolated compact set of periodic points for the flow $\varphi_0$ of $\xi_0$, and 
we define the index of $\varphi$ with respect to $C$ to be the index of $\varphi_0$ with respect to $C$. 

In the next step, we start with any $G$-manifold of finite orbit type. For a prime number $p$, $M^p$ is a $G$-manifold with the diagonal action, and $\setZ_p$ acts freely on the 
set $M^p\setminus\Delta_f$, where 
\[
\Delta_f=\{(x_1,\dots, x_p)\in M^p\;|\;\exists\;1\leq i<j\leq p:\;x_i=x_j\} 
\]
is the thickened diagonal. We denote the quotient manifold under this action by $M_p$. A flow $\varphi$ on $M$ induces a flow $\varphi_p$ on $M_p$ by acting diagonally. For a 
periodic point $(x, T)\in M\times\setR$, we write $(x_p, \frac Tp)$ for the equivalence class of 
\[
\left((x, \varphi(x, \frac Tp), \varphi(x, \frac{2T}p),\dots, \varphi(x, \frac{(p-1)T}p)), \frac Tp\right)
\]
in $M_p\times\setR$. Taking $p$ to be a large prime number, such that $\frac{kT}p$ is not a period of the point $x$ for any $k=1, \dots, p-1$, we see that $(x_p, \frac Tp)$ is
a periodic point of $\varphi_p$. Moreover it can be shown, compare \cite{franzosa}, that if $C$ is compact, $p$ can be chosen in a way such that $\frac{kT}p$ is not a period for
any periodic point $(x, T)\in C$. We therefore can define, for $p$ large, the set $C_p$ to be the set consisting of all $(x_p, \frac Tp)$ with $(x, T)\in C$.

We remark that there is a one-to-one correspondence of periodic points of $\varphi$ with the periodic points of $\varphi_p$. Indeed, if $(x, T)$ is a periodic point of $\varphi$,
then $(x_p, \frac Tp)$ clearly is a periodic point of $\varphi_p$. Conversely, if $([x_1,\dots, x_p], T)$ is any periodic point of $\varphi_p$, then 
\[
 [\varphi(x_1, T),\dots, \varphi(x_p, T)]=[(x_1, \dots, x_p)],
\]
and so there is an index $j$ such that $\varphi(x_{j+i}, T)=x_i$, $i=1, \dots, p$, where we calculate modulo $p$ in the index. It follows that
\[
 x_1=\varphi(x_{i\cdot j+1}, i\cdot T).
\]
In particular it follows that, $(x_1, p\cdot T)$ is a periodic point of $\varphi$.

With all this preparation, we have a homological index defined for the flow $\varphi_p$ with respect to $C_p$, for $p$ sufficiently large. It is an element of 
$\calH_1^G({M_p}_+)$. 

We now pass to the group
\[
 \calH^G(M)=\quot{\prod_{p\sklein{\mbox{ prime }}}\calH_1^G({M_p}_+)}{\bigoplus_{p\sklein{\mbox{ prime }}}\calH_1^G({M_p}_+)}.
\]
The various homological indices define an equivariant homological Fuller index of the flow $\varphi$ with respect to $C$, which we write as $I(\varphi)=I(\varphi, C)$.

The restriction maps define homomorphisms
\[
 \psi_K:\calH^G(M)\to\quot{\prod_{p\sklein{\mbox{ prime }}}H_1({M^K_p}_+)}{\bigoplus_{p\sklein{\mbox{ prime }}}H_1({M^K_p}_+)}.
\]
The latter group admits, via covering space theory, a map
\[
 \mu_K:\quot{\prod_{p\sklein{\mbox{ prime }}}H_1({M^K_p}_+)}{\bigoplus_{p\sklein{\mbox{ prime }}}H_1({M^K_p}_+)}\to
       \quot{\prod_{p\sklein{\mbox{ prime }}}\setZ_p}{\bigoplus_{p\sklein{\mbox{ prime }}}\setZ_p}.
\]
We denote the group on the right by $\calZ$. 

Next, we recall from \cite{may}, that the Burnside ring of a compact Lie group admits an inclusion $A(G)\to C(\Phi G, \setZ)$. Here, $\Phi G$ is the set of conjugacy classes of 
subgroups of $G$ with finite Weyl group. It is topologized as the quotient of a subset of the space of subgroups, which carries the topology induced by the Hausdorff metric. In 
this topology, $\Phi G$ is compact and totally disonnected. In particular, $C(\Phi G, \setZ)$ consists of the locally constant functions. 

The composition maps $\mu_K\circ\psi_K$ constitute an element in $C(\Phi G, \calZ)$, namely the element
\[
 \Phi G\to\calZ,\;(H)\mapsto\mu_K\circ\psi_K(I(\varphi)).
\]
$C(\Phi G, \calZ)$ is isomorphic to $C(\Phi G, \setZ)\tensor\calZ$. The image of the homological Fuller index $I(\varphi)$ under the map $\mu_K\circ\psi_K$ is the homological 
Fuller index of the fixed point flow $\varphi^K$. This follows as in the case of the Lefschetz number by applying restriction to the defining diagram of the homological
indices. Franzosa has shown that this element is actually a rational number, identified via the embedding
\[
 \setQ\to\calZ,\;\frac rq\mapsto\{a_p\}_{p\sklein{\mbox{ prime}}},\;r=a_p\cdot q\mod p.
\]
It follows that the image of the homological equivariant Fuller index $I(\varphi)\in\calH^G(M)$ in $C(\Phi G, \setZ)\tensor\calZ$ actually constitutes an element in 
$C(\Phi G, \setZ)\tensor\setQ$. By Lemma 2.10 of \cite{may}, this ring is isomorphic to the rationalized Burnside ring $A(G)\tensor\setQ$. 

We define the equivariant Fuller index $F_G(\varphi)$ to be the image of the homological equivariant Fuller index in the rationalized Burnside ring via this identification.

The following result is now straightforward. 

\begin{theorem}
The equivariant Fuller index $F_G$ is a $G$-homotopy invariant of a flow $\varphi$ with respect to an isolated set $C$ of periodic points, and has the following properties.
 \begin{enumerate}[i)]
  \item It takes values in the rationalized Burnside ring $A(G)\tensor\setQ$.
  \item If $C$ consists of finitely many periodic orbits $\gamma_1,\dots, \gamma_n$ and $\varphi_i$ is the flow $\varphi$, restricted to an isolating neighbourhood of the 
  orbit $\gamma_i$, then
  \[
   F_G(\varphi)=\sum_{i=1}^nF_G(\varphi_i).
  \]

  \item If $\varphi$ has a single periodic orbit of multiplicity $m$, then $F_G(\varphi)=L_G(P^m)\tensor\frac 1m\in A(G)\tensor\setQ$, where $P$ is an equivariant Poincar\'{e}
  map for the orbit, considered with multiplicity one.
 
  \item If $\eta_H(F_G(\varphi))\neq0$, then $\varphi$ has a periodic orbit of orbit type at least $(H)$.
 \end{enumerate}
\end{theorem}

\begin{proof}
 i) is immediately clear and iii) follows from ii), the equivariant Lefschetz fixed point theorem \ref{thm:lefschetz} and the fact that every equivariant flow is $G$-homotopic to
 a flow with finitely many $G$-orbits of periodic orbits, see \cite{wruck}. So it remains to prove ii). The evident non-equivariant analogue has been proven by Franzosa in 
 \cite{franzosa}. We use some basic theory of equivariant Poincar\'{e} maps as in e.g. \cite{wruck}. $P^H$ is a Poincar\'{e} map for the fixed point flow $\varphi^H$, so
 by Franzosas result, $F_G(\varphi)$ restricts under $\eta_H$ to $F(\varphi^H)=\frac{L((P^H)^m)}m$. Clearly, $L_G(P^m)\tensor\frac1m$ restricts to the same element. Therefore, 
 the two elements are equal in $A(G)\tensor\setQ$. 
\end{proof}

\end{document}